\documentclass[twocolumn]{autart}    

\usepackage{amsmath}
\usepackage{color}
\usepackage{amssymb}
\usepackage{graphicx}
\usepackage{subcaption}
\usepackage{cite}

\newtheorem{assumption}{Assumption}

\newtheorem{theorem}{Theorem}

\newtheorem{remark}{Remark}
\newenvironment{proof}
{\textbf{{Proof} }}
{$\blacksquare$}

\allowdisplaybreaks

\begin{document}

\begin{frontmatter}

\title{Exponential and Prescribed-Time Extremum Seeking\\ with Unbiased Convergence }

\author[tugrul]{Cemal Tugrul Yilmaz}\ead{cyilmaz@ucsd.edu}, 
\author[tugrul]{Mamadou Diagne}\ead{mdiagne@ucsd.edu}\ and
\author[tugrul]{Miroslav Krstic}\ead{krstic@ucsd.edu}

\address[tugrul]{Department of Mechanical and Aerospace Engineering, University of California, San Diego, La Jolla, CA 92093-0411, USA}


\begin{abstract}                          
We present multivariable extremum seeking (ES) designs that achieve unbiased convergence to the optimum. Two designs are introduced: one with exponential unbiased convergence (unbiased extremum seeker, uES) and the other with user-assignable prescribed-time unbiased convergence (unbiased PT extremum seeker, uPT-ES).  In contrast to the conventional ES, which uses persistent sinusoids and results in steady-state oscillations around the optimum,  the exponential uES employs an exponentially decaying amplitude in the perturbation signal (for achieving convergence) and an exponentially growing demodulation signal (for making the convergence unbiased). The achievement of unbiased convergence also entails employing an adaptation gain that is sufficiently large in relation to the decay rate of the perturbation amplitude. Stated concisely, the bias is eliminated by having the learning process outpace the waning of the perturbation. The other algorithm, uPT-ES, employs prescribed-time convergent/blow-up functions in place of constant amplitudes of sinusoids, and it also replaces constant-frequency sinusoids with chirp signals whose frequency grows over time. Among the convergence results in the ES literature, uPT-ES  may be the strongest yet in terms of the convergence rate (prescribed-time) and accuracy (unbiased). To enhance the robustness of uES to a time-varying optimum, exponential functions are modified to keep oscillations at steady state. Stability analysis of the designs is based on a state transformation,  averaging,  local exponential/PT stability of the averaged system, local stability of the transformed system, and local exponential/PT stability of the original system. For numerical implementation of the developed ES schemes and comparison with previous ES designs, the problem of source seeking by a two-dimensional velocity-actuated point mass is considered.
\end{abstract}

\end{frontmatter}

\section{Introduction}
\textbf{A brief historical background.} Extremum Seeking (ES), originally invented in 1922 by Maurice LeBlanc, a French industrialist and inventor \cite{leblanc1922electrification}, gained significant popularity during the mid-20th century.  However, in the 1960s, ES experienced a period of stagnation or slow growth. It wasn't until the development of stability proof in \cite{krstic2000stability} around the year 2000 that the algorithm began to regain momentum. Since then, the field of ES has witnessed remarkable theoretical advancements \cite{durr2013lie}, \cite{ghaffari2012multivariable}, \cite{oliveira2022extremum}, \cite{oliveira2016extremum}, \cite{scheinker2017model}, \cite{tan2010extremum}, and has found practical applications across various domains \cite{ghaffari2014power}, \cite{scheinker2021extremum}, \cite{zhang2007extremum}.

\subsection{Extremum seeking with time-varying perturbation}

Extremum Seeking serves as a valuable method for discovering optimal solutions in systems that are exposed to  uncertainties acting on their   input-output map, offering an alternative to adaptive control. This self-optimizing control approach facilitates the exploration of unknown maps by leveraging  known key properties\cite{draper1951principles,morosanov1957method,pervozvanskii1960continuous}. Thanks  to its model-free nature and convergence guarantees, ES has been a uniquely effective optimization technique for locating and tracking the optima of cost functions associated to static and dynamic systems.
The fundamental principle of ES is to introduce a small perturbation to the system through an excitation signal, observe the system's response, estimate the gradient by demodulating the output, and adjust the system's inputs towards the vicinity of the optima. Due to the persistent excitation present in the process, achieving exact convergence to the extremum cannot be ensured, and  instead, steady-state oscillations around the extremum are commonly observed. These consistent oscillations, which help reactivate the search algorithm in response to changes or deviations in the optimal solution,  ensure robustness with respect to possible drifts of the optima. Indeed,  achieving unbiased convergence is highly challenging  when the exploration, solely driven by the excitation mechanism, occurs in an unknown environment.  Generally speaking,   sustained oscillations can negatively affect both stable and unstable systems' performance in dynamic extremum seeking in that convergence to the system's setpoint is not possible, and at best, limit cycle behavior occurs near the setpoint \cite{scheinker2014non}. For instance, industrial-mechanical systems such as bridge and gantry cranes \cite{sorensen2007controller,d2000exponential,d1994feedback} or cable-payload systems of deep-sea construction vessels \cite{wang2022delay,wang2023delay} face significant challenges in achieving high precision positioning due to limit cycles caused by payload oscillations. These difficulties results in costly and time- and energy-consuming  manoeuvering to satisfy high precision objectives. 

In addressing the steady-state oscillation problem in classical ES, a scheme with a decaying perturbation amplitude is introduced in \cite{tan2009global}. By choosing a sufficiently large initial value of the amplitude, convergence to an arbitrarily small neighborhood of the global extremum is guaranteed  in the presence of local extrema.
This result is followed by \cite{wang2016stability}, which claims exponential and unbiased convergence to local extremum by updating the amplitude based on the system output.
However, this claim has later been proven to be incorrect in \cite{atta2019comment}.
Indeed, as clarified in \cite{atta2019comment}, under certain conditions, the algorithm of \cite{wang2016stability} enables convergence to a point on a manifold that is not necessarily extremum. In the literature, several non-smooth ES designs have been proposed to address the issue of biased convergence.  In \cite{scheinker2014non}, an instance of a non-smooth ES design is presented with the objective of reducing perturbations as the system approaches zero. Subsequently, the authors of \cite{suttner2017exponential} demonstrate exponential convergence to zero. 
In \cite{grushkovskaya2018class}, a formula for the design of ES vanishing at the extremum is provided, unifying and generalizing previous results \cite{scheinker2014non} and \cite{suttner2017exponential}. The asymptotic and exponential
convergence to zero is guaranteed in \cite{grushkovskaya2018class} under some restrictive
assumptions on the cost function. An interesting extension of the framework  in \cite{grushkovskaya2018class}  to dynamic systems can be found in \cite{grushkovskaya2021extremum}. It is important to highlight that all the results in \cite{grushkovskaya2021extremum}, \cite{grushkovskaya2018class}, \cite{scheinker2014non}, and \cite{suttner2017exponential},  are derived assuming that the exact location of the optimum point is unknown, while the corresponding value of the cost function at the optimum is known.
This restrictive assumption is removed in \cite{bhattacharjee2021extremum} and \cite{pokhrel2021control}, which achieve asymptotic convergence to a neighborhood of the extremum with vanishing oscillation by updating the amplitude based on the gradient estimate. 
In order to reduce convergence bias, \cite{lauand2022extremely} introduces filtering techniques and diminishes learning/update gains while maintaining a constant perturbation amplitude. The method in \cite{lauand2022extremely} is termed ``quasi-stochastic approximation" (QSA), wherein stochastic perturbations in stochastic approximation (SA) are replaced with periodic ones, resembling classical ES. Building on \cite{lauand2022extremely}, \cite{lauand2022markovian} further develops QSA by introducing state-dependent probing signals for improved transient performance, despite not guaranteeing unbiased convergence.
An asymptotic convergence, directly to the optimum, is achieved in \cite{abdelgalil2021lie} and \cite{suttner2019extremum} without requiring the knowledge of the cost function. In a nutshell, none of the aforementioned results, which are based on classical and  Lie brackets averaging or QSA, achieve \emph{exponential and unbiased convergence} directly to unknown extremum.

In the present contribution, we propose a new control framework termed \emph{exponential unbiased extremum seeker (uES)} to ensure \emph{exponential and unbiased convergence} by finely upgrading the perturbation signals design of classical averaging-based ES.  The proposed design approach still retains the robustness feature of classical ES design in the sense that it has the ability to perform  adaptation  when the extremum, rather than being stationary, is subject to a relatively smooth time-varying drift: steady-state oscillations are often crucial to preserve the liveliness of the optimization algorithm. Specifically,  under the occurrence of drifting optima, users are provided with flexible design parameters to fine-tune   the perturbation signal by gradually decreasing its amplitude as the output approaches the extremum, which avoids a complete cessation of the adaptation mechanism and guarantees convergence to an arbitrarily small neighborhood of the optimum. The adjustable oscillations mechanism plays a crucial role in establishing the duality between unbiased convergence and robustness. 

\subsection{Prescribed-time extremum seeking}

Recent years have seen the emergence of  ES paradigms that exhibit the remarkable capability to converge to a neighborhood of the extremum within a finite time interval. These advancements have revolutionized the traditional ES algorithm by enhancing its capabilities, leading to significant progress in the areas of finite-time stability, as explored in \cite{guay2021finite}, and fixed-time stability, as developed in \cite{poveda2021non,poveda2021fixed}. 
The key distinction between finite-time and fixed-time stability lies in that the convergence time in fixed-time stability is independent of initial conditions and has a fixed upper bound determined by the system's parameters for any initial conditions whereas the convergence time in finite-time stability is dependent on initial conditions. More demanding than the two previously invoked stability concepts, the notion of fixed-time stability in user-prescribed time, referred to as prescribed-time stability, has been more recently introduced to the control literature by \cite{song2017time} as the strongest among the existing time-constrained control design methodologies. Generally speaking, providing control designers with the ability to pre-assign a terminal time, regardless of initial conditions and system's parameters, empowers them with a substantial degree of control authority over the dynamics of the system. The implementation of prescribed-time extremum seeking (PT-ES) has been achieved by the authors in \cite{todorovski2023practical} and \cite{ctydelayheatPDETAC} using chirp signals and growing time-varying gains. Specifically, \cite{todorovski2023practical} introduces a PT-ES-based source seeking scheme for unicycles while \cite{ctydelayheatPDETAC} develops PT-ES schemes that incorporate compensation techniques for delay, diffusion partial differential equations (PDEs), and wave PDEs.
The basic principle behind prescribed-time stability of nonlinear systems in \cite{song2017time}, is that, the system is driven to its equilibrium using monotonically increasing controller gain functions that blow up at the final time. Since convergence of states occurs faster than  divergence of gains, the input signal remains bounded. However, the outputs of PT-ES designs in \cite{todorovski2023practical} and \cite{ctydelayheatPDETAC} converge to a neighborhood of the extremum while  leading to unbounded control signals. \emph{Our contribution  develops an unbiased PT-ES (uPT-ES) that  converges to the extremum without bias and ensures boundedness of control signals.}

\subsection{Contributions and organization}

The  contributions of this paper consist of   three  ES algorithms: $(i)$ exponential uES with vanishing oscillations and unbiased convergence, $(ii)$ robust exponential ES with adjustable oscillations, but less accurate convergence, $(iii)$ uPT-ES with vanishing oscillations and unbiased convergence in prescribed time. The concept of exponential uES relies on an exponential decay function that reduces the effect of the perturbation signal and the use of its multiplicative inverse, which grows exponentially, to maximize the effect of the demodulation signal multiplied by the high-pass filtered output. Similar to the prescribed-time stabilization concept presented in \cite{song2017time}, the convergence of the filtered output occurs at a faster rate than the divergence of the inverse function and the convergence of the perturbation, keeping the controller bounded. 
For the stability analysis, we transform the system using the exponentially growing function and then apply classical averaging and singular perturbation methods to show the local stability of the transformed system, which in turn implies the local exponential stability of the original system as well as exponential convergence of the output to the extremum with proper choice of  gains. We introduce the robust exponential ES, which is our second algorithm, by modifying the exponential decay function so that it converges to a small value arbitrarily defined.  Our third and final algorithm, uPT-ES, replaces the exponential decay function with a prescribed-time convergent function and employs chirp signals as perturbation/demodulation signals that grow in frequency over time and ultimately diverge to infinity at the terminal. The prescribed-time stability of the closed-loop system is achieved by using a time-dilation transformation, a state transformation, a classical averaging method, and a time-contraction transformation. 

We evaluate our three ES designs numerically by studying the source seeking problem with a two-dimensional velocity actuated point mass, which has been previously solved with a traditional ES in \cite{zhang2007extremum}. \emph{A discovery that stands as a common ground of the proposed unbiased algorithms is  a duality between learning and unbiased convergence: learning must occur at a rate that surpasses the rate of decay of the waning oscillations.}

The conference version of this paper \cite{ctycdcpaper} introduces the exponential uES for static and dynamic maps, which form Section \ref{expExSect} and \ref{expesdynamicsec} of the current paper. This journal paper builds upon that work and introduces two additional contributions. Section \ref{PTESSec} explores the concept of uPT-ES while Section \ref{robustexposection} focuses on the robust exponential ES. The subsequent sections of the paper are organized as follows: Section \ref{sourceseeking} presents our approaches to solving the source seeking problem. The numerical results are presented and analyzed in Section \ref{applicationsection}. The paper concludes with Section \ref{concsection}.

\section{Exponential Unbiased Extremum Seeker for Static Maps} \label{expExSect}

We consider the following optimization problem
\begin{align}
    \max_{\theta \in \mathbb{R}^n} h(\theta),
\end{align}
where $\theta \in \mathbb{R}^n$ is the input, $h \in \mathbb{R}^n \to \mathbb{R}$
is an unknown smooth function. We make the following assumption regarding the unknown static map $h(\cdot)$: 
\begin{assumption} \label{assconvex}
The function $h$ is $\mathcal{C}^4$, and there exists $\theta^{*} \in \mathbb{R}^n$ such that
\begin{align}
    \frac{\partial}{\partial \theta} h(\theta^{*})={}&0, \\
    \frac{\partial^2}{\partial \theta^2} h(\theta^{*})={}&H<0, \quad H=H^T.
\end{align}
\end{assumption}
Assumption \ref{assconvex} guarantees the existence of a maximum of the function $h(\theta)$ at $\theta=\theta^*$. We measure the unknown function $h(\theta)$ in real time as follows
\begin{align}
    y(t)={}&h(\theta(t)), \qquad t \in [t_0,\infty),
\end{align}
in which $y \in \mathbb{R}$ is the output.
Our aim is to design an ES algorithm using output feedback $y(t)$ in order to achieve exponential convergence of $\theta$ to $\theta^{*}$ while simultaneously maximizing the steady state value of $y$, without requiring prior knowledge of either $\theta^*$ or the function $h(\cdot)$. 
\begin{figure}[t]
    \centering
    \includegraphics[width=.9\columnwidth]{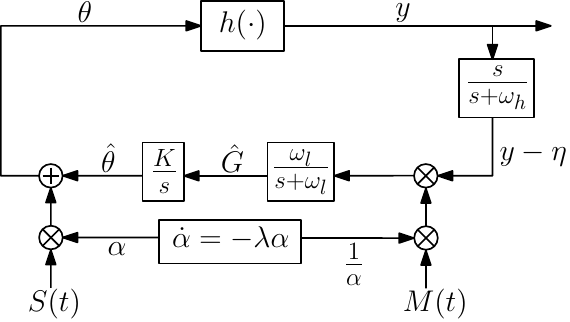} 
    \caption{ Exponential uES scheme. The design uses an exponential decay function $\alpha$ to gradually reduce the effect of the perturbation signal $S(t)$ and its multiplicative inverse $\frac{1}{\alpha}$ to gradually increase the effect of the demodulation signal $M(t)$.}
    \label{ESBlock}
\end{figure}
Our exponential uES design for static maps is schematically illustrated in Fig. \ref{ESBlock}, where $K$ is an $n \times n$ positive diagonal matrix, the filter coefficients $\omega_h$ and $\omega_l$ are positive real numbers, the perturbation and demodulation signals are defined as
\begin{align}
    S(t)=& \begin{bmatrix} a_1 \sin(\omega_1 t)&  & \cdots & & a_n \sin(\omega_n t) \end{bmatrix}^T, \label{St} \\
    M(t)=&\begin{bmatrix} \frac{2}{a_1} \sin(\omega_1 t)&  & \cdots & & \frac{2}{a_n} \sin(\omega_n t) \end{bmatrix}^T, \label{Mt}
\end{align}
respectively and the exponential decay function $\alpha$ is governed by
\begin{align}
    \dot{\alpha}(t)=&{}-\lambda \alpha(t), \qquad \alpha(t_0)=\alpha_0. \label{mu}
\end{align}
The parameters $\alpha_0, \lambda$ are positive real numbers, the amplitudes $a_i$ are real numbers, $\omega_i / \omega_j$ are rational and the frequencies are chosen such that $\omega_i \neq \omega_j$ and $\omega_i + \omega_j \neq \omega_k$  for distinct $i, j$ and $k$. 
We select  the probing frequencies $\omega_i$'s as follows
\begin{align}
    \omega_i={}&\omega {\omega}_i^{\prime}, \qquad i \in \{1, 2, \dots, n \}, \label{param1} 
\end{align}
where $\omega$ is a positive constant and and $\omega_i^{\prime}$ is a rational number.
In addition, the parameters should satisfy the following conditions:
\begin{align}
    {\lambda} <{}& \frac{{\omega_l}}{2}, \frac{{\omega}_h}{2}, \label{cond1} \\
    K >{}& (\omega_l - \lambda) \frac{\lambda}{\omega_l}  \left( \frac{1}{-H} \right) >0.  \label{cond2}
\end{align}
Note that if $K > \frac{\lambda}{-2H}$, stability is achieved for all admissible $\lambda$ (not exceeding $\omega_l/2$). The algorithm can be used without the low-pass filter, in which case these conditions become, taking the limit $\omega_l \rightarrow\infty$, 
\begin{align}
\lambda <{}& \frac{\omega_h}{2}, \\
K >{}& \frac{\lambda}{-H} >0. 
\end{align}
The interpretation of the conditions is that perturbation amplitude $\alpha$ and the demodulation amplitude $1/\alpha$ can decay and grow, respectively, but not too fast, while the estimate $\hat{\theta}$ needs to be updated fast enough, for the given rate of decay/growth of the amplitudes. Learning needs to outpace the waning of the perturbation.

We summarize closed-loop system depicted in Fig. \ref{ESBlock} as follows
\begin{align}
    \frac{d}{d t}\begin{bmatrix}
    \tilde{\theta} \\ \hat{G} \\ \tilde{\eta} \\ \alpha
    \end{bmatrix}=\begin{bmatrix}
        K \hat{G} \\ -\omega_l \hat{G}+\omega_l (y-h(\theta^*)-\tilde{\eta})\frac{1}{\alpha} M(t) \\
        -\omega_h \tilde{\eta}+\omega_h (y-h(\theta^*)) \\
        -\lambda \alpha
    \end{bmatrix}, \label{ffed}
\end{align}
in view of the transformations
\begin{align}
    \tilde{\theta}={}&\hat{\theta}-\theta^*, \label{transf1} \\
    \tilde{\eta}={}&\eta-h(\theta^*), \label{transf2}
\end{align}
where $\eta$ is governed by
\begin{align}
    \dot{\eta}=-\omega_h \eta+\omega_h y. \label{etadef}
\end{align}

The convergence result is stated in the following theorem.
\begin{theorem} \label{the1}
Consider the feedback system \eqref{ffed} with the parameters 
that satisfy \eqref{cond1}, \eqref{cond2} under Assumption \ref{assconvex}. There exists $\bar{\omega}$ and for any $\omega > \bar{\omega}$ there exists an open ball $\mathcal{B}$ centered at the point $(\hat{\theta}, \hat{G}, \eta, \alpha)=(\theta^*, 0, h(\theta^*),0) = : \Upsilon$ such that for any initial condition starting in the ball $\mathcal{B}$, the system \eqref{ffed} has a unique solution and the solution converges exponentially to $\Upsilon$.
Furthermore, $y(t)$ exponentially converges to $h(\theta^*)$.
\end{theorem}

Before presenting the proof, we need to emphasize that the system \eqref{ffed} does not exhibit an equilibrium point in the conventional sense. Theorem \ref{the1} establishes the convergence of the system to a point $\Upsilon$ within the state space. This point, $\Upsilon$, is not an asymptotically or exponentially stable equilibrium; rather, it represents a specific point in the state space towards which the system converges.

\begin{proof} Let us proceed through the proof step by step.

\textbf{Step 1: State transformation.}
Consider the following transformations
\begin{align}
     \tilde{\theta}_f=&{} \frac{1}{\alpha} \tilde{\theta}, \qquad
     \hat{G}_f={} \frac{1}{\alpha} \hat{G},  \qquad 
     \tilde{\eta}_f={} \frac{1}{\alpha^2} \tilde{\eta}, \label{trans}
\end{align}
which transform \eqref{ffed} to the following system
\begin{align}
    &\frac{d}{d t}\begin{bmatrix}
    \tilde{\theta}_f & \hat{G}_f & \tilde{\eta}_f & \alpha
    \end{bmatrix}^T \nonumber \\
    &=\begin{bmatrix}
        \lambda \tilde{\theta}_f+K \hat{G}_f \\ (\lambda-{\omega}_l)\hat{G}_f +{\omega}_l\left[\nu(\tilde{\theta}_f \alpha+{S}(t)\alpha)-\tilde{\eta}_f \alpha^2 \right] \frac{{M}(t)}{\alpha^2}  \\
        (2\lambda-{\omega}_h) \tilde{\eta}_f+{\omega}_h \frac{1}{\alpha^2} \nu(\tilde{\theta}_f \alpha+{S}(t)\alpha) \\
        -{\lambda} \alpha
    \end{bmatrix}, \label{fedtautran}
\end{align}
where
\begin{align}
    \nu(z)=h(\theta^*+z)-h(\theta^*) \label{nudefined}
\end{align}
with $z=\tilde{\theta}_f \alpha+{S}(t)\alpha$ in view of $\theta=\hat{\theta}+S(t)\alpha$ and \eqref{transf1}. From Assumption \ref{assconvex}, we get
\begin{align}
    \nu(0)=0, \quad \frac{\partial }{\partial z}\nu(0)=0, \quad \frac{\partial^2 }{\partial z^2}\nu(0)=H<0. \label{nupartials}
\end{align}


\textbf{Step 2: Verification of the feasibility of \eqref{fedtautran} for averaging.}
We rewrite the system \eqref{fedtautran} in the time scale $\tau=\omega t$ as follows
\begin{align}
    &\frac{d}{d \tau}\begin{bmatrix}
    \tilde{\theta}_f & \hat{G}_f & \tilde{\eta}_f & \alpha
    \end{bmatrix}^T \nonumber \\
    &=\frac{1}{\omega} \begin{bmatrix}
        {\lambda} \tilde{\theta}_f+{K} \hat{G}_f \\ ({\lambda}-{\omega}_l) \hat{G}_f +{\omega}_l \left[\nu(\tilde{\theta}_f \alpha+\bar{S}(\tau)\alpha)-\tilde{\eta}_f \alpha^2 \right] \frac{\bar{M}(\tau)}{\alpha^2}  \\
        (2{\lambda}-{\omega}_h) \tilde{\eta}_f +{\omega}_h \frac{1}{\alpha^2} \nu(\tilde{\theta}_f \alpha+\bar{S}(\tau)\alpha) \\
        -{\lambda} \alpha
    \end{bmatrix},  \label{systautrans}
\end{align}
where $\bar{S}(\tau)=S(\tau/\omega), \bar{M}(\tau)=M(\tau/\omega)$.
Let us write the system \eqref{systautrans} in compact form as 
\begin{align}
    \frac{d{\zeta}_f}{d \tau}={}(1/\omega) \mathcal{F}(\tau,\zeta_f),
\end{align}
where $\zeta_f=\begin{bmatrix} \tilde{\theta}_f & \hat{G}_f & \tilde{\eta}_f & \alpha \end{bmatrix}^T$.
For the application of the averaging theorem in \cite{khalil2002nonlinear},
we need to show that $\mathcal{F}(\tau,\zeta_f)$ and its partial derivatives with respect to $\zeta_f$ up to the second order on compact sets of $\zeta_f$ for all $\tau \geq \omega t_0$ are continuous and bounded.
The proof is trivial for $\mathcal{F}(\tau,\zeta_f)$ excluding the term $\nu(\alpha \tilde{\theta}_f+\alpha \bar{S}(\tau))\frac{1}{\alpha^2}$. To complete the proof, we utilize Taylor's theorem to write
\begin{align}
    \nu(z)={}& \sum_{i=1}^n \sum_{j=1}^n  z_i z_j  \int_0^1 (1-s) \frac{\partial^2 \nu}{\partial z_i \partial z_j}(sz) ds \label{taylortheorem}
\end{align}
in view of \eqref{nupartials}. By substituting $z=\alpha\tilde{\theta}_f+\alpha \bar{S}(\tau)$ into \eqref{taylortheorem} and multiplying both sides by $\frac{1}{\alpha^2}$, we obtain
\begin{align}
\frac{1}{\alpha^2}\nu(&\tilde{\theta}_f\alpha+\bar{S}(\tau)\alpha)={}  \sum_{i=1}^n \sum_{j=1}^n  (\tilde{\theta}_{f_i}+a_i \sin(\omega_i^{\prime} \tau))(\tilde{\theta}_{f_j} +a_j \nonumber \\
&\times \sin(\omega_j^{\prime}\tau))\int_0^1 (1-s) \frac{\partial^2 \nu}{\partial z_i \partial z_j}\left(s \tilde{\theta}_f \alpha+s\bar{S}(\tau)\alpha\right) ds. \label{taylor2}
\end{align}
Next, we apply the mean value theorem to obtain
\begin{align}
\frac{1}{\alpha^2}\nu(\tilde{\theta}_f\alpha&+\bar{S}( \tau)\alpha)={}\frac{1}{2}  \sum_{i=1}^n \sum_{j=1}^n  (\tilde{\theta}_{f_i}+a_i \sin(\omega_i^{\prime}\tau))(\tilde{\theta}_{f_j}  \nonumber \\
&+a_j\sin(\omega_j^{\prime}\tau))\frac{\partial^2 \nu}{\partial z_i \partial z_j}
\left(\mathfrak{s} \tilde{\theta}_f \alpha+\mathfrak{s}\bar{S}(\tau)\alpha\right) \label{aftermvt}
\end{align}
for some $\mathfrak{s} \in [0,1]$. By Assumption \ref{assconvex}, \eqref{aftermvt} is continuous and bounded on compact sets of $\tilde{\theta}_f$ and $\alpha$. Considering the $\mathcal{C}^4$ property of $\nu$ and using the mean value theorem, we prove the continuity and boundedness of the partial derivatives of \eqref{taylor2} with respect to $\tilde{\theta}_f$ and $\alpha$ up to the second order on compact sets of $\tilde{\theta}_f$ and $\alpha$. Therefore, $\mathcal{F}(\tau,\zeta_f)$ satisfies the continuity and boundedness assumptions of the averaging theorem in \cite{khalil2002nonlinear}.

\textbf{Step 3: Averaging operation.} 
Let us define the common period of the probing frequencies as follows
\begin{align}
    \Pi=2\pi \times \text{LCM}\left\{\frac{1}{\omega_i}\right\}, \qquad i \in \{1,2\dots,n\}, \label{Pidef}
\end{align}
where LCM stands for the least common multiple. The average of the system \eqref{systautrans} over the period $\Pi$ is given by
\begin{align}
    \frac{d}{d \tau}\begin{bmatrix}
    \tilde{\theta}_f^a \\ \hat{G}_f^a  \\ \tilde{\eta}_f^a 
    \\ \alpha^a
    \end{bmatrix}^T={}&\frac{1}{\omega} \begin{bmatrix}
        {\lambda} \tilde{\theta}_f^a+{K} \hat{G}_f^a \\ ({\lambda}-{\omega}_l) \hat{G}_f^a \\
        (2{\lambda}-{\omega}_h) \tilde{\eta}_f^a  \\
        -{\lambda} \alpha^a
    \end{bmatrix}  \nonumber \\
    +& \frac{1}{\omega}  \begin{bmatrix} 0 \\ {\omega}_l \frac{1}{\Pi} \int_{0}^{\Pi} \nu(\tilde{\theta}_f^a \alpha^a+\bar{S}(\sigma)\alpha^a)  \frac{\bar{M}(\sigma)}{(\alpha^a)^2} d\sigma \\ {\omega}_h \frac{1}{\Pi} \int_{0}^{\Pi}  \nu(\tilde{\theta}_f^a \alpha^a+\bar{S}(\sigma)\alpha^a) \frac{1}{(\alpha^a)^2} d\sigma \\ 0 \end{bmatrix}, \label{averagesys}
\end{align}
where $\tilde{\theta}_f^a, \hat{G}_f^a, \tilde{\eta}_f^a$ and $\alpha^a$ denote the
average versions of the states $\tilde{\theta}_f, \hat{G}_f, \tilde{\eta}_f$ and $\alpha$, respectively. It follows from \eqref{averagesys} that the average equilibrium denoted as $\begin{bmatrix}
\tilde{\theta}_f^{a,e} & \hat{G}_f^{a,e} & \tilde{\eta}_f^{a,e} & \alpha^{a,e}
\end{bmatrix}^T$ satisfies
\begin{align}
    &{\lambda} \tilde{\theta}_f^{a,e}={}-{K} \hat{G}_f^{a,e}, \nonumber \\
    &\alpha^{a,e}={}0, \nonumber \\
    &({\omega}_l-{\lambda}) \hat{G}_f^{a,e} \nonumber \\
    &\quad ={}\lim_{\alpha^{a,e} \to 0} \left[ \frac{{\omega}_l}{\Pi} \int_0^{\Pi} \nu(\tilde{\theta}_f^{a,e} \alpha^{a,e}+\bar{S}(\sigma)\alpha^{a,e})  \frac{\bar{M}(\sigma)}{(\alpha^{a,e})^2} d\sigma \right], \nonumber \\
    &({\omega}_h-2{\lambda}) \tilde{\eta}_f^{a,e} \nonumber \\
    &\quad ={}\lim_{\alpha^{a,e} \to 0} \left[ \frac{{\omega}_h}{\Pi} \int_0^{\Pi} \nu(\tilde{\theta}_f^{a,e} \alpha^{a,e}+\bar{S}(\sigma)\alpha^{a,e})  \frac{1}{(\alpha^{a,e})^2} d\sigma \right].
\end{align}
By performing a Taylor series approximation of $\nu$ in view of \eqref{nupartials} as follows
\begin{align}
    \nu(z)={}&\frac{1}{2} \sum_{i=1}^n \sum_{j=1}^n \frac{\partial^2 \nu }{\partial z_i \partial z_j}(0) z_i z_j \nonumber \\
    &+\frac{1}{3! } \sum_{i=1}^n \sum_{j=1}^n \sum_{k=1}^n  \frac{\partial^3 \nu }{\partial z_i \partial z_j \partial z_k}(0) z_i z_j z_k+\mathcal{O}(|z|^4) \label{nutaylor}
\end{align}
with $z=\tilde{\theta}_f^{a,e} \alpha^{a,e}+\bar{S}(\sigma)\alpha^{a,e}$, we compute
\begin{align}
   &\lim_{\alpha^{a,e} \to 0} \left[ \frac{1}{\Pi} \int_0^{\Pi} \nu(\tilde{\theta}_f^{a,e} \alpha^{a,e}+\bar{S}(\sigma)\alpha^{a,e})  \frac{\bar{M}(\sigma)}{(\alpha^{a,e})^2} d\sigma \right] \nonumber \\ 
   &={}\lim_{\alpha^{a,e} \to 0} \Bigg[ \frac{1}{\Pi} \int_0^{\Pi} \frac{1}{2}  \sum_{i=1}^n \sum_{j=1}^n \frac{\partial^2 \nu }{\partial z_i \partial z_j}(0)  (\tilde{\theta}_{f_i}^{a,e}+a_i \sin(\omega^{\prime}_i\sigma)) \nonumber \\
   & \times (\tilde{\theta}_{f_j}^{a,e} +a_j \sin(\omega^{\prime}_j \sigma)) (\alpha^{a,e})^2 \frac{\bar{M}(\sigma)}{(\alpha^{a,e})^2} d\sigma+\frac{(\alpha^{a,e})^3}{(\alpha^{a,e})^2} \mathcal{O}(|a|^2) \Bigg], \nonumber \\
   &={}H\tilde{\theta}_f^{a,e},  
\end{align}
and
\begin{align}
     &\lim_{\alpha^{a,e} \to 0} \left[ \frac{1}{\Pi} \int_0^{\Pi} \nu(\tilde{\theta}_f^{a,e} \alpha^{a,e}+\bar{S}(\sigma)\alpha^{a,e})  \frac{1}{(\alpha^{a,e})^2} d\sigma \right] \nonumber \\
     &\hspace{2cm}={}\frac{1}{2} \sum_{i=1}^n \sum_{j=1}^n H_{i,j} \tilde{\theta}_{f_i}^{a,e} \tilde{\theta}_{f_j}^{a,e}+  \frac{1}{4} \sum_{i=1}^n H_{i,i} a_i^2, 
\end{align}
by L'Hospital's rule, where $H_{i,j}=\frac{\partial^2 \nu }{\partial z_i \partial z_j}(0)$ and $\tilde{\theta}_{f_i}^{a,e}$ is the $i$th element of $\tilde{\theta}_{f}^{a,e}$. Then, we obtain the equilibrium of the average system \eqref{averagesys} as
\begin{align}
&\begin{bmatrix}
    \tilde{\theta}_f^{a,e} & \hat{G}_f^{a,e} & \tilde{\eta}_f^{a,e} & \alpha^{a,e}
\end{bmatrix}^T \nonumber \\
    &\qquad =\begin{bmatrix}
    0_{1 \times n} & 0_{1 \times n} & \frac{{\omega}_h}{4({\omega}_h-2{\lambda})} \sum_{i=1}^n H_{i,i} a_i^2 & 0
\end{bmatrix}^T, \label{aveq}
\end{align}
provided that $ {\omega_l}\neq{}{\lambda}, {\omega}_h\neq{}2{\lambda}$ and $K \neq{} {\lambda}({\lambda}-{\omega}_l)  {{\omega}}_l^{-1}H^{-1}$.

\begin{figure*}[ht]
    \centering
    \includegraphics[width=1.5\columnwidth]{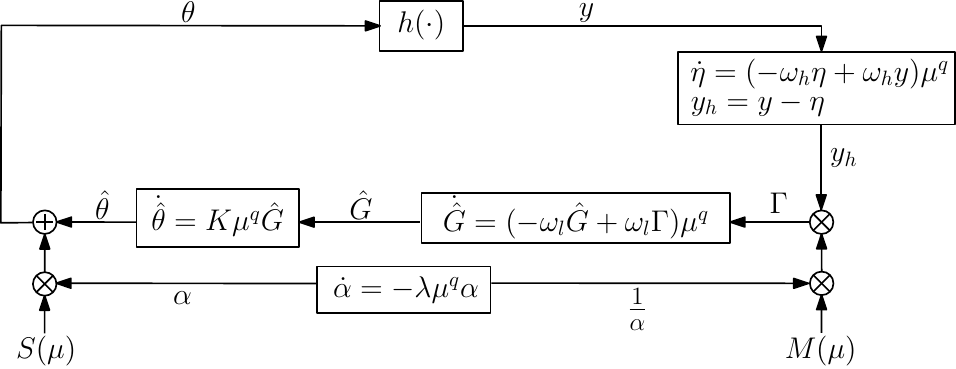}
    \caption{uPT-ES scheme. The design modifies the exponential uES in Fig. \ref{ESBlock} by incorporating $\mu^q$, with $q \geq 1$, into all system dynamics and using hyperbolic chirps in the perturbation and demodulation signals.}
    \label{PTESBlock}
\end{figure*}

\textbf{Step 4: Stability analysis.}
The Jacobian of the average system \eqref{averagesys} at the equilibrium \eqref{aveq} is given by
\begin{align}
    &J_f^a \nonumber \\
    =&\frac{1}{\omega} \begin{bmatrix} {\lambda}I_{n \times n} & {K} & 0_{n \times 1} & 0_{n \times 1}\\ {\omega}_l  H & ({\lambda}-{\omega}_l)I_{n \times n} & 0_{n \times 1} & \frac{{\omega}_l}{\Pi}  \int_0^{\Pi} \frac{\partial \left( \frac{\nu\bar{M}}{(\alpha^{a})^2}\right)}{\partial \alpha^a}  d\sigma 
    \\ 0_{1 \times n} & 0_{1 \times n} & (2{\lambda}-{\omega}_h) & \frac{{\omega}_h}{\Pi}  \int_0^{\Pi} \frac{\partial \left( \frac{\nu}{(\alpha^{a})^2}\right)}{\partial \alpha^a}  d\sigma  \\ 0_{1 \times n} & 0_{1 \times n} & 0 & -{\lambda} \end{bmatrix}.
\end{align}
Note that $J_f^a$ is block-upper-triangular and Hurwitz provided that \eqref{cond1} and \eqref{cond2} are satisfied.
This proves the local exponential stability of the average system \eqref{averagesys}. Then, based on the averaging theorem \cite{khalil2002nonlinear}, we show that there exists $\bar{\omega}$ and for any $\omega > \bar{\omega}$, the system \eqref{systautrans} has a unique exponentially stable periodic solution $(\tilde{\theta}_f^{\Pi}(\tau), \hat{G}_f^{\Pi}(\tau), \tilde{\eta}_f^{\Pi}(\tau), \alpha^{\Pi}(\tau))$ of period $\Pi$ and this solution satisfies
\begin{align}
    \left| \begin{bmatrix} \tilde{\theta}_f^{\Pi}(\tau) \\ \hat{G}_f^{\Pi} (\tau)  \\ \tilde{\eta}_f^{\Pi}(\tau)-\frac{{\omega}_h}{4({\omega}_h-2{\lambda})} \sum_{i=1}^n H_{i,i} a_i^2  \\ \alpha^{\Pi}(\tau)  \end{bmatrix} \right| \leq \mathcal{O}\left( \frac{1}{\omega} \right).
\end{align}
In other words, all solutions $(\tilde{\theta}_f(\tau), \hat{G}_f(\tau),  \tilde{\eta}_f(\tau), \alpha(\tau))$ exponentially converge to an $\mathcal{O}\left( 1/\omega \right)$-neighborhood of the origin. The signal $\alpha(\tau)$, in particular, exponentially converges to zero. Noting this fact and recalling the transformations \eqref{trans}, we can deduce that the system \eqref{ffed} with states $\tilde{\theta}(t), \hat{G}(t), \tilde{\eta}(t)$ has a unique solution and is exponentially stable at the origin. 
In particular, based on \eqref{trans}, both $\tilde{\theta}(t)$ and $\hat{G}(t)$ exhibit exponential convergence to zero at a rate of $\lambda$, while $\tilde{\eta}(t)$ converges to zero exponentially with a rate of $2\lambda$.

\textbf{Step 5: Convergence to extremum.} Considering the results in Step 4 and recalling from \eqref{trans} and Fig. \ref{ESBlock} that 
\begin{align}
    \theta(t)={}&\alpha(t) \tilde{\theta}_f(t)+\theta^*+\alpha(t) S(t),
\end{align}
we conclude the exponential convergence of $\theta(t)$ to $\theta^*$  at the rate of $\lambda$. Taking into account the boundedness of $\tilde{\theta}_f(t)$ for all $t \geq t_0$, we deduce from \eqref{nudefined} and \eqref{aftermvt} the exponential convergence of $y(t)=h(\theta(t))$ to $h(\theta^*)$ at the rate of $2\lambda$. This completes the proof of Theorem \ref{the1}. 
\hfill \end{proof}

According to Theorem \ref{the1}, the initial value of the decaying function $\alpha(t)$ should be sufficiently close to zero, which limits the range of possible values for $\alpha_0$ in \eqref{mu}. This can also be interpreted as the requirement for the amplitudes to be small enough in traditional extremum-seeking approaches.

\section{Unbiased Prescribed-Time Extremum Seeker for Static Maps} \label{PTESSec}

In this section, our aim is to design an ES algorithm for static maps, which guarantees the unbiased convergence of $\theta$ to $\theta^{*}$ within a terminal time $T$, where the time $T$ is prescribed by user a priori. 
Our uPT-ES design is schematically illustrated in Fig. \ref{PTESBlock}. Different from the exponentially convergent $\alpha$-dynamics \eqref{mu}, we define a prescribed-time convergent $\alpha$-dynamics as follows
\begin{align}
    \dot{\alpha}(t)=&{}-\lambda \mu^q(t-t_0) \alpha(t), \qquad \alpha(t_0)=\alpha_0 \label{mudot}
\end{align}
with $q \geq 1$ and the following smooth function
\begin{align}
    \mu(t-t_0)={}\frac{T}{T+t_0-t}, \quad t \in [t_0,t_0+T). \label{xidef}
\end{align}
The solution of \eqref{mudot} is given by
\begin{align}
    \alpha(t)={}&\alpha(t_0) e^{-\lambda\int_{t_0}^t \mu^q(\sigma-t_0)d\sigma} \nonumber \\
    ={}&\begin{cases}
    \alpha_0 \mu^{-\lambda T}(t-t_0),& \text{if } q= 1,\\
    \alpha_0 e^{- \frac{\lambda T}{q-1} \left(\mu^{q-1}(t-t_0)-1\right)},  & \text{if } q> 1 
\end{cases} \label{mudef}
\end{align}
with the property that $\alpha(T)=0$. Note that the growth of the $\mu$-signal in the second condition of \eqref{mudef} increases as $q$ is increased beyond $1$, resulting in a faster decay of the $\alpha$-signal. 

We summarize closed-loop system depicted in Fig. \ref{PTESBlock} as follows
\begin{align}
    \frac{d}{d t}\begin{bmatrix}
    \tilde{\theta} \\ \hat{G} \\ \tilde{\eta} \\ \alpha
    \end{bmatrix}=\begin{bmatrix}
        K \mu^q \hat{G} \\ -\omega_l \mu^q \hat{G}+\omega_l \mu^q (y-h(\theta^*)-\tilde{\eta}) \frac{1}{\alpha} M(\mu) \\
        -\omega_h \mu^q \tilde{\eta}+\omega_h \mu^q (y-h(\theta^*)) \\
        -\lambda \mu^q \alpha
    \end{bmatrix}, \label{ptclosed}
\end{align}
in view of the transformations \eqref{transf1}, \eqref{transf2} where $\eta$ is governed by
\begin{align}
    \dot{\eta}=(-\omega_h  \eta+\omega_h y)\mu^q.
\end{align}

We define the following dilation and contraction transformations
\begin{align}
    \check{\tau}={}&\begin{cases}
    t_0+T\ln(\mu), & \text{if } q= 1,\\
    t_0+T\left( \frac{\mu^{q-1}-1}{q-1}  \right),  & \text{if } q> 1, 
\end{cases}  \label{dilat}  \\
    t={}&\begin{cases}
    t_0+T\left(1-e^{-\frac{\check{\tau}-t_0}{T}}\right), & \text{if } q= 1,\\
    t_0+T\left(1-\left(\frac{T}{T+(q-1)(\check{\tau}-t_0)} \right)^{\frac{1}{q-1}} \right), & \text{if } q> 1, 
\end{cases} \label{contrac}
\end{align}
for $\check{\tau} \in [t_0,\infty)$, $t \in [t_0,t_0+T)$. To achieve PT convergence to the extremum, we replace the sinusoids with “chirpy” perturbation and demodulation signals whose frequency grows rather than being constant:
\begin{align}
    &S(\mu) \nonumber \\
    &={}\begin{cases}
    \Big[ a_1 \sin(\omega_1 (t_0+T\ln(\mu))  \, \,  \cdots    \\
    \hspace{0.2cm} \cdots \, \,  a_n \sin(\omega_n (t_0+T\ln(\mu)) \Big]^T, & \text{if } q= 1, \\
    \Big[ a_1 \sin\left(\omega_1  \left(t_0+T \frac{\mu^{q-1}-1}{q-1}  \right)  \right) \, \, \cdots \\
    \hspace{0.2cm} \cdots \, \,  a_n \sin\left(\omega_n  \left(t_0+T \frac{\mu^{q-1}-1}{q-1}  \right)  \right) \Big]^T, & \text{if } q> 1, 
\end{cases}  \\
    &M(\mu) \nonumber \\
    &={}\begin{cases}
    \Big[ \frac{2}{a_1} \sin(\omega_1 (t_0+T\ln(\mu))  \, \,  \cdots    \\
    \hspace{0.2cm} \cdots \, \,  \frac{2}{a_n} \sin(\omega_n (t_0+T\ln(\mu)) \Big]^T, & \text{if } q= 1, \\
    \Big[ \frac{2}{a_1} \sin\left(\omega_1  \left(t_0+T \frac{\mu^{q-1}-1}{q-1}  \right)  \right) \, \, \cdots \\
    \hspace{0.2cm} \cdots \, \,  \frac{2}{a_n} \sin\left(\omega_n  \left(t_0+T \frac{\mu^{q-1}-1}{q-1}  \right)  \right) \Big]^T, & \text{if } q> 1. 
\end{cases}
\end{align}

The convergence result is stated in the following theorem.
\begin{theorem} \label{the2}
Consider the feedback system \eqref{ptclosed} with the parameters that satisfy \eqref{cond1}, \eqref{cond2} under Assumption \ref{assconvex}. There exists 
$\bar{\omega}$ and for any $\omega > \bar{\omega}$ 
there exists an open ball $\mathcal{B}$ centered at the point $(\hat{\theta}, \hat{G}, \eta, \alpha)=(\theta^*, 0, h(\theta^*),0) = : \Upsilon$ such that  for any initial condition starting in the ball $\mathcal{B}$, the system \eqref{ptclosed} has a unique solution and the solution converges to $\Upsilon$ in prescribed time $T$. Furthermore, $y(t)$ converges to $h(\theta^*)$ in prescribed time $T$.
\end{theorem}
\begin{proof}

\textbf{Step 0: Time dilation from $t$ to $\check{\tau}$.}
Considering \eqref{dilat}, \eqref{contrac} along with
\begin{align}
    \frac{d \check{\tau}}{dt}={}\mu^q(t-t_0), \label{dtaudt}
\end{align}
we can write the system \eqref{ptclosed} in the dilated $\check{\tau}$-domain 
\begin{align}
    \frac{d}{d \check{\tau}}\begin{bmatrix}
    \tilde{\theta} \\ \hat{G} \\ \tilde{\eta} \\ \alpha
    \end{bmatrix}=\begin{bmatrix}
        K  \hat{G} \\ -\omega_l \hat{G}+\omega_l  (y-h(\theta^*)-\tilde{\eta}) \frac{1}{\alpha}\mathcal{M}(\check{\tau})  \\
        -\omega_h \tilde{\eta}+\omega_h (y-h(\theta^*)) \\
        -\lambda \alpha
    \end{bmatrix}, \label{ptclosedtaucheck}
\end{align}
with the following perturbation/demodulation signals
\begin{align}
    \mathcal{S}(\check{\tau})=& \begin{bmatrix} a_1 \sin(\omega_1 \check{\tau})&  & \cdots & & a_n \sin(\omega_n \check{\tau}) \end{bmatrix}^T, \label{Sttau} \\
    \mathcal{M}(\check{\tau})=&\begin{bmatrix} \frac{2}{a_1} \sin(\omega_1 \check{\tau})&  & \cdots & & \frac{2}{a_n} \sin(\omega_n \check{\tau}) \end{bmatrix}^T. \label{Mttau}
\end{align}

Note that the system \eqref{ptclosedtaucheck} with \eqref{Sttau}, \eqref{Mttau} is in the similar form to \eqref{ffed} with \eqref{St}, \eqref{Mt}, except that the dilated time $\check{\tau}$ is used instead of $t$. 
The utilization of this temporal transformation technique facilitates the application of the averaging theorem because the frequency of perturbation/demodulation signals becomes constant in $\check{\tau}$-domain. The remainder of the proof can be completed similarly to the 
proof of Theorem \ref{the1} by following steps from $1$ to $5$ and performing time contraction from $\check{\tau}$ to $t$.
\hfill
\end{proof}

\section{Exponential Unbiased Extremum Seeker for Dynamic Systems} \label{expesdynamicsec}
In this section, we extend our results in Section \ref{expExSect} to dynamic systems. For this, we consider a general multi-input single-output nonlinear model
\begin{align}
    \dot{x}={}&f(x,u), \\
    y={}&h(x),
\end{align}
where $x \in \mathbb{R}^m$ is the state, $u \in \mathbb{R}^n$ is the input, $y \in \mathbb{R}$ is the output  and the unknown functions $f: \mathbb{R}^m \times \mathbb{R}^n \to \mathbb{R}^m$ and $h: \mathbb{R}^m \to \mathbb{R}$ are smooth.  Suppose there is a smooth control law
\begin{align}
    u={}&\phi(x,\theta)
\end{align}
parametrized by a vector parameter $\theta \in \mathbb{R}^n$. The closed-loop system 
\begin{align}
    \dot{x}={}f(x,\phi(x,\theta)) \label{closedf}
\end{align}
then has equilibria parameterized by $\theta$. We make the following assumptions about the closed-loop system:

\begin{assumption} \label{ass:smm}
There exists a smooth function $l: \mathbb{R}^n \to \mathbb{R}^m$ such that $f (x, \phi(x, \theta))=0$ if and only if $x = l(\theta)$.
\end{assumption}

\begin{assumption} \label{ass:exps}
For each $\theta \in \mathbb{R}^n$, the equilibrium $x = l(\theta)$ of the
system \eqref{closedf} is locally exponentially stable uniformly
in $\theta$.
\end{assumption}

\begin{assumption} \label{ass:exist}
The function $h \circ l$ is $\mathcal{C}^4$, and there exists $\theta^* \in \mathbb{R}^n$ such that
\begin{align}
    \frac{\partial }{\partial \theta} (h \circ l)(\theta^*)={}&0, \\
    \frac{\partial^2 }{\partial \theta^2} (h \circ l)(\theta^*)={}&H<0, \quad H=H^T.
\end{align}
\end{assumption}

We aim to design a controller $u$ to drive the output $y$ directly to its optimum $h \circ l(\theta^*)$ exponentially without any steady-state oscillation and without the need for knowledge of $\theta^*$, $h$, or $l$.

The perturbation and demodulation signals are defined by \eqref{St} and \eqref{Mt}, respectively, and $\alpha$ is governed by \eqref{mu}. The probing frequencies $\omega_i$'s, the filter coefficients $\omega_h$ and $\omega_l$, gain $K$ and parameter $\lambda$ are selected as  follows
\begin{align}
    \omega_i={}&\omega {\omega}_i^{\prime} ={} \mathcal{O}(\omega), \qquad i \in \{1, 2, \dots, n \}, \label{dparam1} \\
    \omega_h={}&\omega {\omega}_H={}\omega \delta \omega_H^{\prime} ={} \mathcal{O}(\omega \delta ), \\
    \omega_l={}&\omega {\omega}_L={}\omega \delta \omega_L^{\prime} ={} \mathcal{O}(\omega \delta ), \\
    K={}&\omega {K}^{\prime}={}\omega \delta {K}^{\prime \prime} ={}\mathcal{O}(\omega \delta), \\
    \lambda={}&\omega {\lambda}^{\prime}={}\omega \delta {\lambda}^{\prime \prime}={}\mathcal{O}(\omega \delta), \label{dparam5}
\end{align}
where $\omega$ and $\delta$ are small positive constants, $ {\omega}_i^{\prime}$ is a rational number, ${\omega}_H^{\prime}, {\omega}_L^{\prime}$ and ${\lambda}^{\prime \prime}$ are $\mathcal{O}(1)$ positive constants, ${K}^{\prime \prime}$ is a $n \times n$ diagonal
matrix with $\mathcal{O}(1)$ positive elements. In addition, the parameters should satisfy \eqref{cond1} and \eqref{cond2}.

We summarize the closed-loop system as follows
\begin{align}
    \frac{d}{d t}\begin{bmatrix} x \\
    \tilde{\theta} \\ \hat{G} \\ \tilde{\eta} \\ \alpha
    \end{bmatrix}=\begin{bmatrix} f(x,\phi(x,\theta^*+\tilde{\theta}+S(t))) \\
        K \hat{G} \\ -\omega_l \hat{G}+\omega_l (y-h \circ l (\theta^*)-\tilde{\eta}) \frac{1}{\alpha} M(t) \\
        -\omega_h \tilde{\eta}+\omega_h (y-h\circ l (\theta^*)) \\
        -\lambda \alpha
    \end{bmatrix}.  \label{closeddynamical}
\end{align}

The convergence result is stated in the following theorem.
\begin{theorem}
Consider the feedback system \eqref{closeddynamical} under Assumptions
\ref{ass:smm}--\ref{ass:exist}. There exists $\bar{\omega}>0$ and for any $\omega \in (0, \bar{\omega})$ there exists $\bar{\delta}>0$ such that for the given $\omega$ and $\delta \in (0, \bar{\delta})$ there exists an open ball $\mathcal{B}$ centered at the point $(x, \hat{\theta} , \hat{G}, \eta) = (l(\theta^*), \theta^*, 0, h \circ l(\theta^*))= : \Upsilon$ such that for any initial condition starting in the ball $\mathcal{B}$, the system \eqref{closeddynamical} has a unique solution and the solution converges exponentially to $\Upsilon$.
Furthermore, $y(t)$ exponentially converges to $h \circ l(\theta^*)$. 
\end{theorem}

\begin{proof} Let us proceed through the proof step by step.

\textbf{Step 1: Time-scale separation.}
We rewrite the system \eqref{closeddynamical} in the time scale $\tau=\omega t$ as
\begin{align}
    \omega \frac{\text{d} x}{\text{d} \tau}={}&f(x,\phi(x,\theta^*+\tilde{\theta}+\bar{S}(\tau)\alpha)), \label{systau1} \\
    \frac{d}{d \tau}\begin{bmatrix} 
    \tilde{\theta} \\ \hat{G} \\ \tilde{\eta} \\ \alpha
    \end{bmatrix}={}&\delta \begin{bmatrix}
        K^{\prime \prime} \hat{G} \\ -\omega_L^{\prime} \hat{G}+\omega_L^{\prime} (y-h\circ l(\theta^*)-\tilde{\eta}) \frac{1}{\alpha} \bar{M}(\tau) \\
        -\omega_H^{\prime} \tilde{\eta}+\omega_H^{\prime} (y-h\circ l(\theta^*)) \\
        -\lambda^{\prime \prime} \alpha
    \end{bmatrix}, \label{systau2}
\end{align}
where $\bar{S}(\tau)=S(\tau/\omega), \bar{M}(\tau)=M(\tau/\omega)$.

\textbf{Step 2: State transformation.} Consider the following transformations
\begin{align}
     \tilde{\theta}_f={}& \frac{1}{\alpha} \tilde{\theta}, \qquad
     \hat{G}_f={} \frac{1}{\alpha} \hat{G},  \qquad 
     \tilde{\eta}_f={} \frac{1}{\alpha^2} \tilde{\eta},
\end{align}
which transform \eqref{systau1}, \eqref{systau2} to the following system
\begin{align}
    \omega \frac{\text{d} x}{\text{d} \tau}={}&f(x,\phi(x,\theta^*+\tilde{\theta}_f\alpha+\bar{S}(\tau)\alpha)), \label{sysftau1} \\
    \frac{d \zeta_{f}}{d \tau} 
    ={}&\delta E(\tau, x, \zeta_f),  \label{sysftau2}
\end{align}
where $\zeta_{f}={}\begin{bmatrix}
    \tilde{\theta}_{f} & \hat{G}_{f} & \tilde{\eta}_{f} & \alpha
    \end{bmatrix}^T$ and
\begin{align}
    &E(\tau, x, \zeta_f) \nonumber \\
    &\, \,={}\begin{bmatrix}
        \lambda^{\prime \prime} \tilde{\theta}_f+K^{\prime \prime} \hat{G}_f \\ (\lambda^{\prime \prime}-{\omega}_L^{\prime})\hat{G}_f +{\omega}_L^{\prime}(y-h\circ l(\theta^*)-\tilde{\eta}_f\alpha^2) \frac{1}{\alpha^2} \bar{M}(\tau)  \\
        (2\lambda^{\prime \prime}-{\omega}_H^{\prime}) \tilde{\eta}_f+{\omega}_H^{\prime} \frac{1}{\alpha^2} (y-h\circ l(\theta^*)) \\
        -{\lambda}^{\prime \prime} \alpha
    \end{bmatrix}.
\end{align}

\textbf{Step 3: Averaging analysis for reduced system.} We first freeze $x$ in \eqref{sysftau1} at its equilibrium value $x =L(\tau,\zeta_f)=l(\theta^*+\tilde{\theta}_f\alpha+\bar{S}(\tau)\alpha)$, substitute it into \eqref{sysftau2} and then get the reduced system
\begin{align}
    \frac{d \zeta_{f,r}}{d \tau}={}\delta E(\tau, L(\tau,\zeta_{f,r}),\zeta_{f,r}), \label{zetafrred}
\end{align}
where $\zeta_{f,r}=\begin{bmatrix}
    \tilde{\theta}_{f,r} & \hat{G}_{f,r} & \tilde{\eta}_{f,r} & \alpha
    \end{bmatrix}^T$,
\begin{align}
    &E(\tau, L(\tau,\zeta_{f,r}),\zeta_{f,r}) \nonumber \\
    &=\begin{bmatrix}
        \lambda^{\prime \prime} \tilde{\theta}_{f,r}+K^{\prime \prime} \hat{G}_{f,r} \\ (\lambda^{\prime \prime}-{\omega}_L^{\prime})\hat{G}_{f,r} +{\omega}_L^{\prime}(\nu(\tilde{\theta}_{f,r}\alpha +\bar{S}(\tau)\alpha)-\tilde{\eta}_{f,r}\alpha^2) \frac{\bar{M}(\tau)}{\alpha^2}   \\
        (2\lambda^{\prime \prime}-{\omega}_H^{\prime}) \tilde{\eta}_{f,r}+{\omega}_H^{\prime} \frac{1}{\alpha^2} \nu(\tilde{\theta}_{f,r}\alpha +\bar{S}(\tau)\alpha) \\
        -{\lambda}^{\prime \prime} \alpha
    \end{bmatrix}
\end{align}
and
\begin{align}
    \nu(z)=h \circ l(z+\theta^*)-h \circ l(\theta^*) \label{nuzetadyn}
\end{align}
with $z=\tilde{\theta}_{f,r}\alpha +\bar{S}(\tau)\alpha$. 
From Assumption \ref{ass:exist}, we get
\begin{align}
    \nu(0)=0, \quad \frac{\partial }{\partial z}\nu(0)=0, \quad \frac{\partial^2 }{\partial z^2}\nu(0)=H<0. \label{nupartialsdyn}
\end{align}
Note that the reduced system \eqref{zetafrred} has the same structure as \eqref{systautrans} except the different constant parameters. Therefore, we can perform averaging analysis and stability analysis in Step 3 and 4 of the proof of Theorem 1, respectively, for the reduced system \eqref{zetafrred}. Then, we conclude that 
there exists $\delta$ such that for all $\delta \in (0, \bar{\delta})$,
the system \eqref{zetafrred} has a unique exponentially stable periodic solution  $\zeta_{f,r}^{\Pi}(\tau)= \begin{bmatrix}
    \tilde{\theta}_{f,r}^{\Pi}(\tau) & \hat{G}_{f,r}^{\Pi}(\tau) & \tilde{\eta}_{f,r}^{\Pi}(\tau) & \alpha^{\Pi}(\tau)
    \end{bmatrix}^T$ such that
\begin{align}
    \frac{d \zeta_{f,r}^{\Pi}(\tau)}{d \tau}={}\delta E(\tau, L(\tau,\zeta_{f,r}^{\Pi}(\tau)),\zeta_{f,r}^{\Pi}(\tau)).
\end{align}

\textbf{Step 4: Singular perturbation analysis.} To convert the system \eqref{sysftau1} and \eqref{sysftau2} into the standard singular perturbation form, we shift the states $\zeta_{f}$ and $x$ using the transformations $\tilde{\zeta}_{f}=\zeta_{f}-\zeta_{f,r}^{\Pi}(\tau)$ and $\tilde{x}=x-L(\tau,\zeta_f)$ such that
\begin{align}
    \frac{d \tilde{\zeta}_f}{d \tau}={}&\delta \tilde{E}(\tau,\tilde{x},\tilde{\zeta}_f), \label{dzetasing1} \\
    \omega \frac{d \tilde{x}}{d \tau}={}&\tilde{F}(\tau,\tilde{x},\tilde{\zeta}_f), \label{dxsing}
\end{align}
where
\begin{align}
    \tilde{E}(\tau,\tilde{x},\tilde{\zeta}_f)={}&E(\tau, \tilde{x}+L(\tau,\tilde{\zeta}_f+ \zeta_{f,r}^{\Pi}(\tau)), \tilde{\zeta}_f+\zeta_{f,r}^{\Pi}(\tau)) \nonumber \\
    &-E(\tau, L(\tau,\zeta_{f,r}^{\Pi}(\tau)), \zeta_{f,r}^{\Pi}(\tau)), \\
    \tilde{F}(\tau,\tilde{x},\tilde{\zeta}_f)={}&f\Big(\tilde{x}+L(\tau,\tilde{\zeta}_f+ \zeta_{f,r}^{\Pi}(\tau)), \nonumber \\
    &\phi\big(\tilde{x}+L(\tau,\tilde{\zeta}_f+ \zeta_{f,r}^{\Pi}(\tau)), \nonumber \\
    &\theta^*+\tilde{\zeta}_{f_1}\alpha+{\zeta}_{f_1,r}^{\Pi}(\tau)\alpha+\bar{S}(\tau)\alpha\big)\Big),
\end{align}
where $\tilde{\zeta}_{f_1}=\tilde{\theta}_{f}-\tilde{\theta}_{f,r}^{\Pi}(\tau)$ and ${\zeta}_{f_1,r}^{\Pi}(\tau)=\tilde{\theta}_{f,r}^{\Pi}(\tau)$. Note that $\tilde{x} = 0$ is the quasi-steady state. 
By substituting the quasi-steady state into \eqref{dzetasing1}, we obtain the following reduced model  
\begin{align}
    \frac{d \tilde{\zeta}_{f,r}}{d \tau}={}\delta \tilde{E}(\tau,0,\tilde{\zeta}_{f,r}), \label{zetareduced}
\end{align}
which has an equilibrium at the origin $\tilde{\zeta}_{f,r}=0$. We prove in Step 3 that this equilibrium is exponentially stable.

The next step in the singular perturbation analysis is to examine
the boundary layer model in the time scale $t = \tau/\omega$ as follows
\begin{align}
    \frac{\text{d} x_{\text{b}}}{\text{d} t}={}&\tilde{F}(\tau,x_b,\tilde{\zeta}_f),\nonumber \\
    ={}&f(x_{\text{b}}+l(\theta),\phi(x_{\text{b}}+l(\theta),\theta)). \label{xbdyn}
\end{align}
Recalling $f (l(\theta), \phi(l(\theta), \theta)) \equiv 0$,
we deduce that $x_{\text{b}} \equiv 0$ is an equilibrium of \eqref{xbdyn}. According to Assumption 2, this
equilibrium is locally exponentially stable uniformly in $\theta$.

By combining exponential stability of the reduced model \eqref{zetareduced}
with the exponential stability of the boundary layer model \eqref{xbdyn},
and noting that $\tilde{E}(\tau,0,0)=0, \tilde{F}(\tau,0,0)=0$, we conclude from Theorem 11.4 of \cite{khalil2002nonlinear} that $\tilde{\zeta}_f \to 0$ and $\tilde{x} \to 0$, i.e., $\zeta_f \to \zeta_{f,r}^{\Pi}$ and $x \to l(\theta)=L(\tau, \zeta_f)$ exponentially as $\tau \to \infty$. 

\textbf{Step 5: Convergence to extremum.} Note that $\tilde{\theta}_f (\tau) \to \tilde{\theta}_f^{\Pi} (\tau)$ and $\alpha \to 0$ exponentially. It follows then that $\theta(\tau) = \theta^* + \tilde{\theta}_f (\tau) \alpha + \bar{S}(\tau)\alpha \to \theta^*$ exponentially and $l(\theta)=l(\theta^*+\tilde{\theta}_f\alpha+\bar{S}(\tau)\alpha) \to l(\theta^*)$ exponentially.  Consequently, $y = h(x)$ exponentially converges to $h \circ l(\theta^*)$.
\hfill
\end{proof}

\section{Robust Exponential Extremum Seeker} \label{robustexposection}
An important aspect of our design depicted in Fig. \ref{ESBlock} is that the multiplicative inverse of the function $\alpha$ experiences exponential growth, while the high-pass filtered signal $y-\eta$ decays to zero at a much faster rate, resulting in a bounded signal. However, in practical implementations, the boundedness of the resulting signal may not be guaranteed due to various factors such as measurement noise which prevents the complete convergence of $y-\eta$ to zero, and numerical inaccuracies that may arise from the multiplication of large and small values. The aforementioned limitation is not specific to our design but applies  to any given prescribed-time  stabilization scheme available in the literature \cite{song2017time}, \cite{song2023prescribed}. Furthermore, there may be instances where the extremum point changes over time, rather than being stationary, in which case the traditional extremum seeking design is capable of tracking it. To overcome these challenges and enhance the robustness of our design, we propose 
a modified $\alpha$-signal that exponentially converges to an arbitrarily defined small positive number $\beta$ rather than exponential decay function \eqref{mu}. In this case, our design offers a convergence to the neighborhood of the extremum with adjustable steady state oscillations. The new $\alpha$-signal  is governed by the following dynamics
\begin{align}
    \dot{\alpha}(t)=&{}-\lambda \alpha(t)+\lambda \beta, \qquad \alpha(t_0)=\alpha_0. \label{mumodif}
\end{align}

\begin{remark}
Our new design may seem to boil down to the traditional ES in \cite{krstic2000stability}. However, an ES design that employs a constant, small $\alpha(t)\equiv \beta$ for all $t \geq t_0$ in Figure \ref{ESBlock} results in an initial rapid growth of $\hat{\theta}$. Our design, with adjustable oscillations, addresses this issue. A further discussion is provided in Section \ref{applicationsection}.
\end{remark}

By relaxing the $\mathcal{C}^4$ condition in Assumption \ref{assconvex}, we make the following assumption:
\begin{assumption} \label{convexrobustass}
There exists $\theta^{*} \in \mathbb{R}^n$ such that
\begin{align}
    \frac{\partial}{\partial \theta} h(\theta^{*})={}&0, \\
    \frac{\partial^2}{\partial \theta^2} h(\theta^{*})={}&H<0, \quad H=H^T.
\end{align}
\end{assumption}
We summarize the closed-loop system depicted in Fig. \ref{ESBlock} with the modified  $\alpha$-dynamics \eqref{mumodif} as follows
\begin{align}
    \frac{d}{d t}\begin{bmatrix}
    \tilde{\theta} \\ \hat{G} \\ \tilde{\eta} \\ \alpha
    \end{bmatrix}=\begin{bmatrix}
        K \hat{G} \\ -\omega_l \hat{G}+\omega_l (y-h(\theta^*)-\tilde{\eta})\frac{1}{\alpha} M(t) \\
        -\omega_h \tilde{\eta}+\omega_h (y-h(\theta^*)) \\
        -\lambda \alpha+\lambda \beta
    \end{bmatrix}. \label{closedrobust}
\end{align}
We present the following result for static maps, which can be easily extended to dynamic systems.

\begin{theorem} \label{theoremrobust}
Consider the feedback system \eqref{closedrobust} 
under Assumption \ref{convexrobustass}. There exist $\bar{\omega}, \bar{a}>0$ such that for all $\omega > \bar{\omega}$ and $\beta |a| \in (0, \bar{a})$ 
there exists an open ball $\mathcal{B}$ centered at the point $(\hat{\theta}, \hat{G}, \eta, \alpha)=(\theta^*, 0, h(\theta^*),\beta) = : \Upsilon$ such that  for any initial condition starting in the ball $\mathcal{B}$, the system \eqref{closedrobust} has a unique solution and the solution converges exponentially to an $\mathcal{O}(\beta \delta+ \beta | a|)$-neighborhood of $\Upsilon$.
Furthermore, $y(t)$ exponentially converges to an $\mathcal{O}(\beta^2 \delta^2+ \beta^2 | a|^2)$-neighborhood of $h(\theta^*)$.
\end{theorem}
The proof of this result is given in Appendix \ref{theoremappendix}.

\section{Source Seeking by a Velocity-Actuated Point Mass} \label{sourceseeking}
In this section, we investigate the problem of source localization using an autonomous vehicle modeled as a point mass in a two-dimensional plane
\begin{align}
    \dot{x}_1={}v_{x_1}, \qquad \dot{x}_2={}v_{x_2}, \label{velactmass}
\end{align}
in which the vehicle's position is represented by the vector $\begin{bmatrix} x_1 & x_2 \end{bmatrix}^T$ and its velocity is controlled by inputs $v_{x_1}$ and $v_{x_2}$. The objective of this problem is to guide the vehicle towards the static source of a scalar signal in an environment where the vehicle's position data is not available. The only information provided to the vehicle at its current location is the strength of the signal, which is assumed to decrease as the distance from the source increases. Our specific goal is to detect the source while continuously measuring the source signal, ultimately bringing the vehicle to a complete stop at the exact location of the source. We give a block diagram in Fig. \ref{ptesvehicle}, in which we can apply our exponential uES and uPT-ES designs by using their corresponding $\mu$ and $\alpha$ functions.

\begin{figure}[t]
    \centering
    \includegraphics[width=\columnwidth]{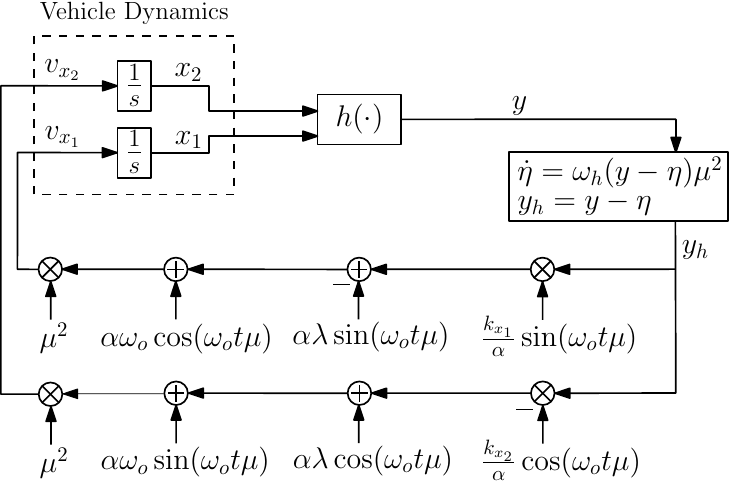} \\
    \caption{The developed ES scheme for velocity-actuated point mass. For exponential convergence, choose $\mu \equiv 1, \alpha=\alpha_0 e^{-\lambda t}$ $\forall t \in [0,\infty)$ and for prescribed-time convergence, choose $\mu=\frac{T}{T-t}, \alpha=\alpha_0 e^{\lambda T \left(1-\frac{T}{T-t}\right)}$ $\forall t \in [0,T)$ with $\alpha_0, \lambda>0$.}
    \label{ptesvehicle}
\end{figure}

For simplicity, but without loss of generality, we assume that the nonlinear map is quadratic with diagonal Hessian matrix, i.e.,
\begin{align}
    h(x_1,x_2)={}h^*-q_{x_1}(x_1-x_1^*)^2-q_{x_2}(x_2-x_2^*)^2, \label{mapquad}
\end{align}
where $(x_1^*,x_2^*)$ is the unknown maximizer, $h^*=h(x_1^*,x_2^*)$ is the unknown maximum, and $q_{x_1}$, $q_{x_2}$ are some unknown positive constants.

Before presenting our results, let us first introduce the new coordinates
\begin{align}
    \tilde{x}_1={}&x_1-x_1^*-\alpha \sin(\omega_o t \mu), \\
    \tilde{x}_2={}&x_2-x_2^*+\alpha \cos(\omega_o t \mu), \\
    \tilde{\eta}={}&\eta-h(x_1^*,x_2^*),
\end{align}
where the signal $\eta$ is defined in \eqref{etadef}. For exponential stability, we can choose $\mu \equiv 1 $ and $\alpha$ as in \eqref{mu} for all $ t \in [0,\infty)$. For prescribed-time stability, we can choose $\mu $ and $\alpha$ as in \eqref{xidef}, \eqref{mudef}, respectively, for all $ t \in [0,T)$.
Then, we summarize the system in Fig. \ref{ptesvehicle} as follows
\begin{align}
    \frac{d}{d t}\begin{bmatrix}
    \tilde{x}_1 \\ \tilde{x}_2 \\ \tilde{\eta} \\ \alpha
    \end{bmatrix}=\begin{bmatrix}
        +k_{x_1} \mu^2 \sin(\omega_o t \mu) (y-h^*-\tilde{\eta})(1/\alpha) \\
        -k_{x_2} \mu^2 \cos(\omega_o t \mu) (y-h^*-\tilde{\eta})(1/\alpha) \\
        -\omega_h \mu^2 \tilde{\eta}+\omega_h  \mu^2 (y-h^*) \\
        -\lambda \mu^2 \alpha
    \end{bmatrix} \label{vehmatrix}
\end{align}
with the parameters chosen as
\begin{align}
    \omega_h>{}&2 \lambda, \qquad k_{x_i}>{}\lambda/q_{x_i}, \qquad i=1,2. \label{paramptsource}
\end{align}

An extension of the exponential uES result in Theorem \ref{the1} as well as the uPT-ES result in Theorem \ref{the2} to the system \eqref{vehmatrix} can be easily done. For convenience, we give the uPT-ES result for this source seeking problem below without a proof.

 \begin{theorem}
Consider the feedback system \eqref{vehmatrix} with
the parameters that satisfy \eqref{paramptsource} and with the nonlinear map of the form \eqref{mapquad}. There exists $\bar{\omega}_o$ and for any $ \omega_o> \bar{\omega}_o$,
there exists an open ball $\mathcal{B}$ centered at the point $(x_1, x_2, \eta, \alpha)=(x_1^*, x_2^*, h^*,0) = : \Upsilon$ such that  for any initial condition starting in the ball $\mathcal{B}$, the system \eqref{vehmatrix} has a unique solution and the solution converges to $\Upsilon$ in prescribed time $T$.
Hence, $y(t)$ converges to $h(x_1^*,x_2^*)$ in prescribed time $T$. Furthermore, the velocity inputs remain bounded over $[0, T)$.
 \end{theorem}

The main contribution of the developed uPT-ES technique in Fig. \ref{ptesvehicle} 
is that we achieve to drive the vehicle directly to the source $h^*$ by keeping the velocities $v_{x_1}, v_{x_2}$ bounded whereas the PT-ES technique in \cite{ctydelayheatPDETAC}
(which can be suitably modified to fit the structure shown in Figure \ref{ptesvehicle} with $\lambda=0$, $\alpha \equiv \alpha_0$ and $\mu$-function \eqref{xidef}) achieves the convergence of the vehicle to a small neighborhood the source with the velocities growing unbounded. The boundedness of the velocities $v_{x_1}, v_{x_2}$ in our design follows from the fact that the signal $y-\eta$ converges to zero in prescribed time proportionally to $\alpha$ as well as the fact that the square of the blow-up function $\mu$ is multiplied by the function $\alpha$ which decays to zero faster, resulting in boundedness and convergence of the velocities to zero, i.e.,
\begin{align}
    \lim_{t \to T} \left( \alpha(t) \mu(t)^2  \right)={}\lim_{t \to T} \left(\alpha_0 \frac{T^2}{(T-t)^2}  e^{\lambda T \left(1-\frac{T}{T-t}\right)}\right)={}0.
\end{align}

\section{Application to Source Seeking Problem} \label{applicationsection}

We consider the application of the developed ES techniques to the problem of source seeking by a velocity-actuated point mass as defined in Section \ref{sourceseeking}.
The velocity of the vehicle in Fig. \ref{ptesvehicle} is controlled by the following ES controllers with appropriate $\mu$ and $\alpha$ functions:
\begin{itemize}
    \item \textbf{Controller 1.} Nominal ES with $\mu \equiv 1$, $\alpha \equiv \alpha_0$, (i.e., $\lambda=0$). This design boils down to Fig. 1 in \cite{zhang2007extremum}, in which the vehicle asymptotically converges to a neighborhood of the source and shows steady-state oscillations around it.
    \item \textbf{Controller 2.} Exponential uES with $\mu \equiv 1$, $\alpha$-function dynamics \eqref{mu}. This design is a modified version of Fig. \ref{ESBlock}.
    \item \textbf{Controller 3.} Robust Exponential ES with $\mu \equiv 1$, $\alpha$-function dynamics \eqref{mumodif} and additional terms $\lambda \beta \sin(\omega_o t)$ in $v_{x_1}$ as well as $-\lambda \beta \cos(\omega_o t)$ in $v_{x_2}$. This design is a modified version of Fig. \ref{ESBlock} with $\alpha$-dynamics \eqref{mumodif}.
    \item \textbf{Controller 4.} uPT-ES with $\mu$-function \eqref{xidef}, $\alpha$-function \eqref{mudef}. This design is a modified version of Fig. \ref{PTESBlock} with $q=2$.    
\end{itemize}

\begin{figure}[t]
    \centering
    \includegraphics[trim=0cm 0cm 0cm 0cm, clip=true, width=\linewidth]{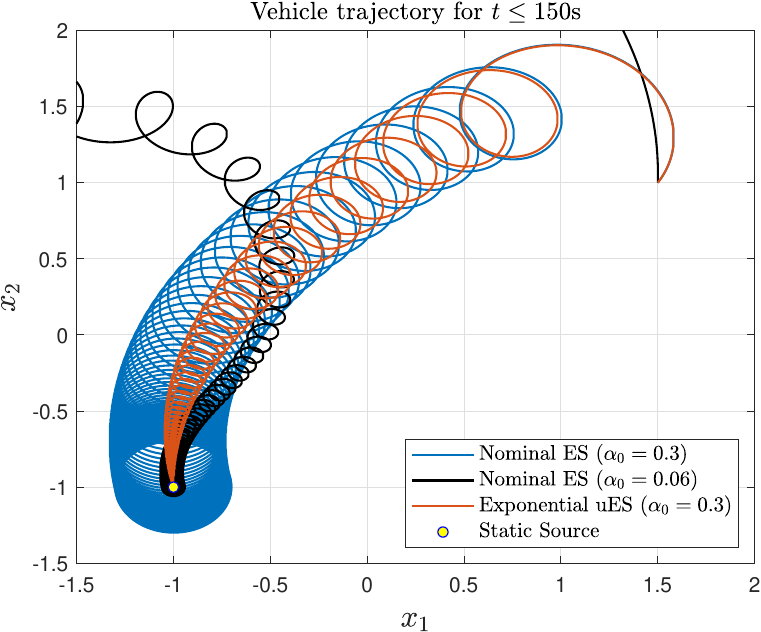}  
    \caption{Static source seeking by an autonomous vehicle \eqref{velactmass}. The nominal ES with low amplitude $\alpha_0$  approaches the source more closely but requires high initial velocity, leading to initial deviation from the source. The exponential uES \eqref{ffed}, with its exponentially decaying amplitude, avoids this issue.}
    \label{staticsourceexpfig}
\end{figure}

\begin{figure}[t]
    \centering
    \includegraphics[trim=0cm 0cm 0cm 0cm, clip=true, width=\linewidth]{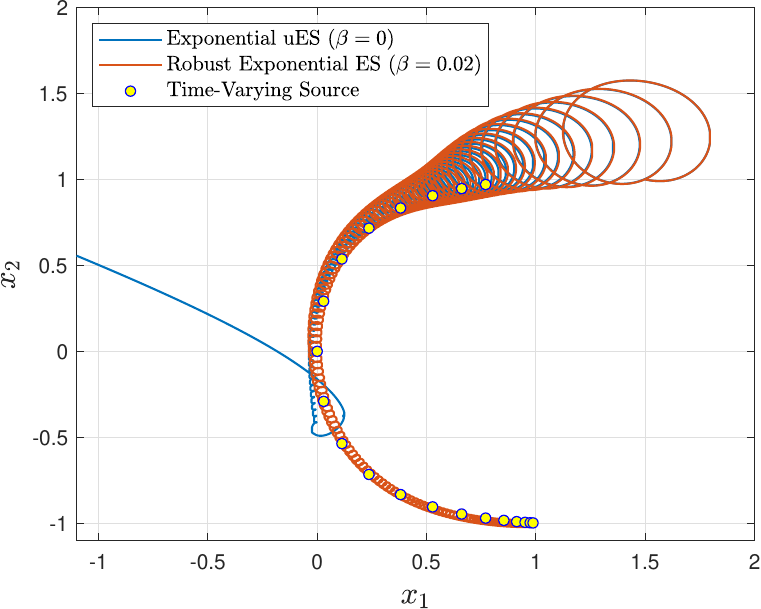}
    \caption{Time-varying source seeking by an autonomous vehicle \eqref{velactmass}. The trajectory of the source is illustrated by yellow circles at 10 second intervals. The robust exponential design \eqref{closedrobust} successfully tracks the source, owing to its amplitude that decays but does not vanish, while exponential design \eqref{ffed} fails to track.}
    \label{staticsourceexprobustfig}
\end{figure}

We now present the results of four numerical simulations to demonstrate the performance of Controllers 1-4 for the source-seeking problem. The real-time measurement is defined as $y(t) = h(x_1(t),x_2(t))$, where the function $h(\cdot)$ is described in \eqref{mapquad} and its parameters are chosen as $(x_1^*,x_2^*)=(-1,-1), h^*=1, q_{x_1}=1, q_{x_2}=0.5$.

\begin{figure}[t]
\begin{subfigure}{\linewidth}
    \centering
    \includegraphics[trim=0cm 0cm 0cm 0cm, clip=true, width=\linewidth]{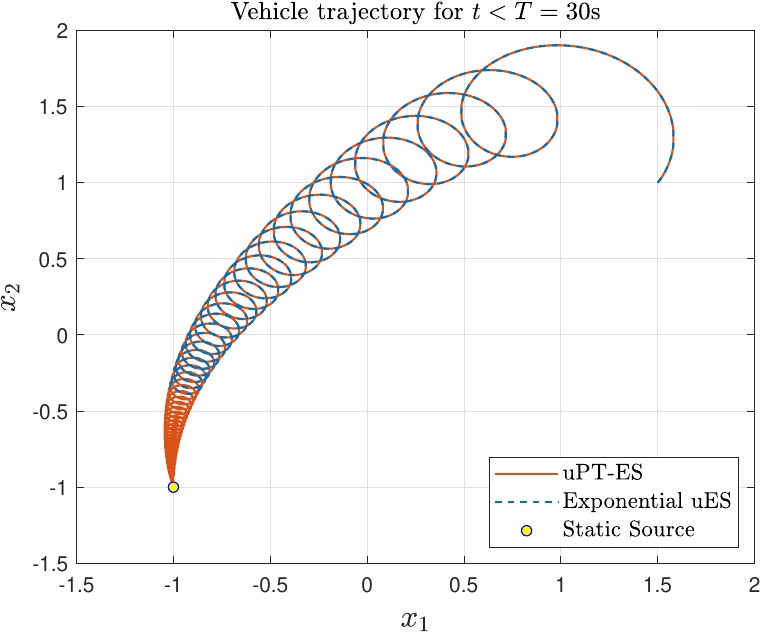}
    \caption{}
    \label{staticsourceptfiga}
\end{subfigure}
\begin{subfigure}{\linewidth}
    \centering
    \includegraphics[trim=0cm 0cm 0cm 0cm, clip=true, width=\linewidth]{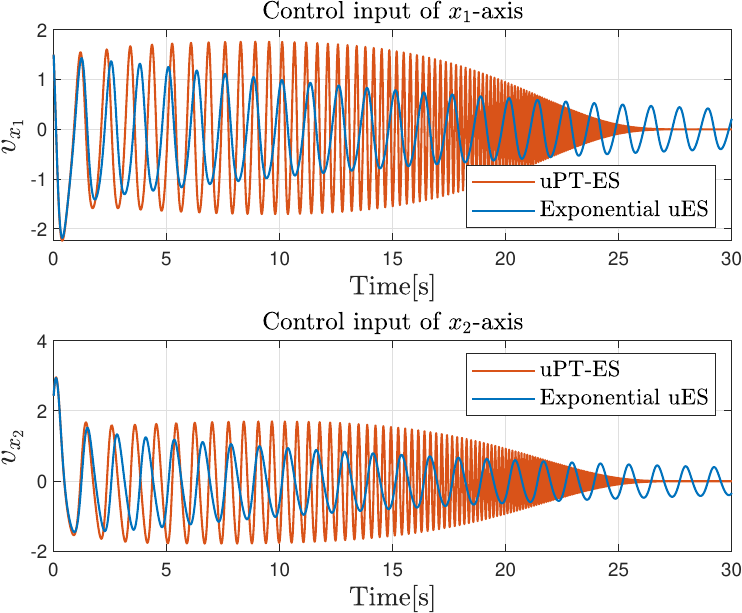}
    \caption{}
    \label{staticsourceptfigb}
\end{subfigure}
\caption{Static source seeking for $t \in [0,30)$ in seconds. (a) Both designs follow the same trajectory, but the exponential uES \eqref{ffed} falls behind the uPT-ES \eqref{ptclosed}. (b) The velocity inputs of the uPT-ES \eqref{ptclosed} exhibit more rapid changes compared to those of the exponential uES \eqref{ffed}.}
\label{staticsourceptfig}
\end{figure}

\textbf{Controller 1 and 2.} The first simulation compares the performance of the nominal ES and exponential uES for a static source. The parameters in Fig. \ref{ptesvehicle} are selected as $\omega_h=1$, $k_{x_1}=k_{x_2}=0.1, \lambda=0.045$ and $\omega_o=5$. The comparison between the nominal ES and exponential uES is shown in Fig. \ref{staticsourceexpfig}. The exponential uES exhibits exponential convergence to the source at $(-1,-1)$ with circular trajectories and exponentially decaying amplitude. On the other hand, the nominal ES with constant amplitude asymptotically converges to the vicinity of the source and shows steady-state oscillation around it. The nominal design with $\alpha_0=0.06$ converges closer to the source than the one with $\alpha_0=0.3$, but it initially moves away from the source due to its high initial velocity. Low initial velocity and perfect convergence are achieved through our design.

\textbf{Controller 2 and 3.} The second simulation examines the case of a time-varying source. The source is modeled as
\begin{align}
    x_1^*(t)={}&1-e^{-0.0003(t-70)^2}, \label{x1start} \\
    x_2^*(t)={}&-\tanh(0.03(t-70)). \label{x2start}
\end{align}
Note that the Gaussian function in \eqref{x1start} and the hyperbolic tangent function in \eqref{x2start} can be referred to as saturating activation functions. The performance of the exponential uES and the robust exponential ES for tracking the time-varying source is shown in Fig. \ref{staticsourceexprobustfig}. The parameters are set as $\beta=0.02, k_{x_1}=0.2, k_{x_2}=0.3$ with $\alpha_0, \omega_h, \omega_o$ being the same as the first simulation. The robust design demonstrates robustness to non-stationary sources with adjustable amplitude, while the exponential uES fails to achieve convergence.

\begin{figure}[t]
\begin{subfigure}{\linewidth}
    \centering
    \includegraphics[trim=0cm 0cm 0cm 0cm, clip=true, width=\linewidth]{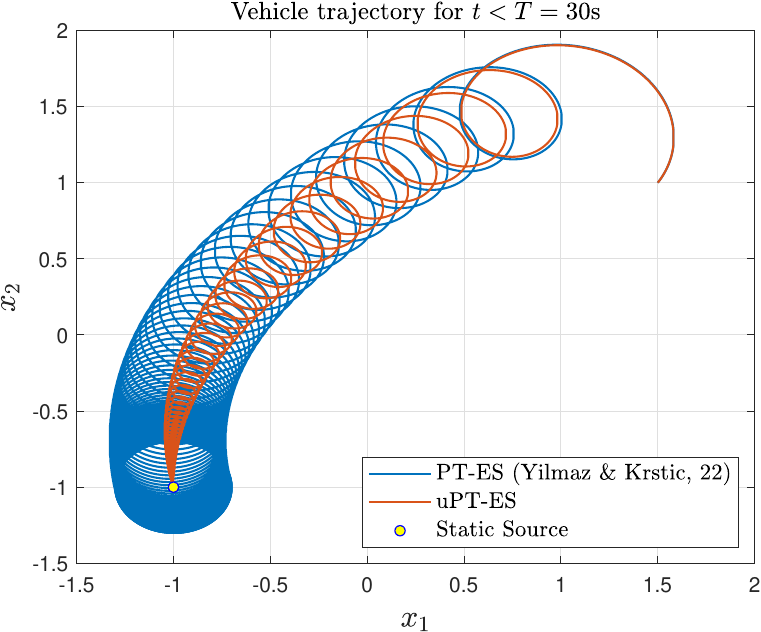}
    \caption{}
    \label{ptcomptodusa}
\end{subfigure}
\begin{subfigure}{\linewidth}
    \centering
    \includegraphics[trim=0cm 0cm 0cm 0cm, clip=true, width=\linewidth]{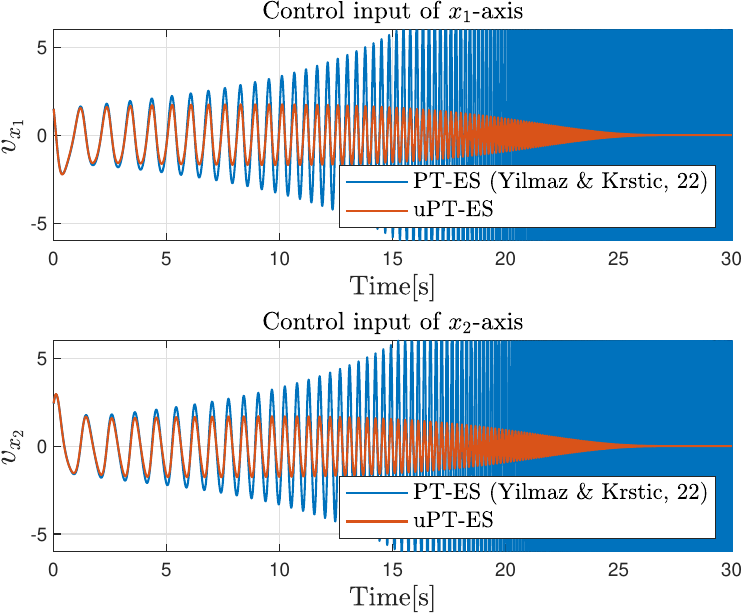}
    \caption{}
    \label{ptcomptodusb}
\end{subfigure}
\caption{
Static source seeking for $t \in [0,30)$ in seconds. (a) The convergence of our uPT-ES \eqref{ptclosed} is compared with the one of delay-free PT-ES in \cite{ctydelayheatPDETAC},
which is implemented on the 2D vehicle illustrated in Fig. \ref{ptesvehicle} by choosing $\lambda=0$, $\alpha\equiv\alpha_0$, and $\mu$ as in \eqref{xidef}.
(b) The velocity inputs in our design are kept bounded, while the inputs in \cite{ctydelayheatPDETAC} grow unbounded.}
\label{ptcomptodus}
\end{figure}

\textbf{Controller 2 and 4.} The third simulation considers the static source problem and implements both the exponential uES design and uPT-ES design for comparison. The final time is set to $T=30s$, and the rest of the parameters are chosen as in the first simulation. As shown in Fig. \ref{staticsourceptfiga}, both designs track the same trajectory, but the exponential uES cannot reach the source in prescribed time.
This is because the uPT-ES design performs more rapid velocity changes,
resulting in bounded but high acceleration in the vehicle, as seen in Fig. \ref{staticsourceptfigb}.

\textbf{Controller 4 and \cite{ctydelayheatPDETAC}.} In our four and final simulation, 
we illustrate the difference between these two different prescribed-time ES designs. The delay-free design in \cite{ctydelayheatPDETAC} can be implemented into Fig. \ref{ptesvehicle} by choosing $\lambda=0, \alpha \equiv \alpha_0$, and $\mu$ as in  \eqref{xidef}.
The final time is set to $T=30s$, and the rest of the parameters are the same as in the first simulation. 
The PT-ES in \cite{ctydelayheatPDETAC} basically improves the convergence of the nominal ES as demonstrated in Fig. \ref{ptcomptodusa}. However, we can see in Fig. \ref{ptcomptodusb} that this comes with the cost of velocities growing unbounded. Our design, on the other hand, achieves unbiased convergence in prescribed time with bounded velocities.

\section{Conclusion} \label{concsection}
In this paper, we address the issue of steady-state oscillations in classical ES. We develop accelerated ES algorithms that not only eliminate the steady-state oscillations, but also achieves unbiased convergence to the extremum exponentially and in prescribed time by employing proper time-varying functions in the perturbation and demodulation stage of the ES loop. For tracking non-stationary optima, we introduce a robust ES scheme with user-adjustable oscillations that gradually decrease but remains non-zero. We evaluate the performance of our ES algorithms on a source-seeking problem. With bounded velocity inputs, our uPT-ES design accurately locates the source, while the delay-free PT-ES algorithm in \cite{ctydelayheatPDETAC} converges to a neighborhood of the source with unbounded input growth. In our  future works,  extension of our classical averaging-based ES results to the ES based on Lie-bracket approximation will be developed.


\begin{thebibliography}{99}

\bibitem{abdelgalil2021lie}
M.~Abdelgalil and H.~Taha.
\newblock Lie bracket approximation-based extremum seeking with vanishing input oscillations.
\newblock {\em Automatica}, 133:109735, 2021.

\bibitem{atta2019comment}
K.~T. Atta and M.~Guay.
\newblock Comment on “on stability and application of extremum seeking control without steady-state oscillation”[automatica 68 (2016) 18--26].
\newblock {\em Automatica}, 103:580--581, 2019.

\bibitem{bhattacharjee2021extremum}
D.~Bhattacharjee and K.~Subbarao.
\newblock Extremum seeking control with attenuated steady-state oscillations.
\newblock {\em Automatica}, 125:109432, 2021.

\bibitem{d1994feedback}
B.~d'Andr{\'e}a Novel, F.~Boustany, F.~Conrad, and B.~P. Rao.
\newblock Feedback stabilization of a hybrid pde-ode system: Application to an overhead crane.
\newblock {\em Mathematics of Control, Signals and Systems}, 7:1--22, 1994.

\bibitem{draper1951principles}
C.~S. Draper and Y.~T. Li.
\newblock Principles of optimalizing control systems and an application to the internal combusion engine.
\newblock 1951.

\bibitem{durr2013lie}
H.-B. D{\"u}rr, M.~S. Stankovi{\'c}, C.~Ebenbauer, and K.~H. Johansson.
\newblock Lie bracket approximation of extremum seeking systems.
\newblock {\em Automatica}, 49(6):1538--1552, 2013.

\bibitem{d2000exponential}
B.~d’Andr{\'e}a Novel and J.-M. Coron.
\newblock Exponential stabilization of an overhead crane with flexible cable via a back-stepping approach.
\newblock {\em Automatica}, 36(4):587--593, 2000.

\bibitem{ghaffari2012multivariable}
A.~Ghaffari, M.~Krsti{\'c}, and D.~Ne{\v{s}}i{\'c}.
\newblock Multivariable newton-based extremum seeking.
\newblock {\em Automatica}, 48(8):1759--1767, 2012.

\bibitem{ghaffari2014power}
A.~Ghaffari, M.~Krsti{\'c}, and S.~Seshagiri.
\newblock Power optimization and control in wind energy conversion systems using extremum seeking.
\newblock {\em IEEE transactions on control systems technology}, 22(5):1684--1695, 2014.

\bibitem{grushkovskaya2021extremum}
V.~Grushkovskaya and C.~Ebenbauer.
\newblock Extremum seeking control of nonlinear dynamic systems using lie bracket approximations.
\newblock {\em International Journal of Adaptive Control and Signal Processing}, 35(7):1233--1255, 2021.

\bibitem{grushkovskaya2018class}
V.~Grushkovskaya, A.~Zuyev, and C.~Ebenbauer.
\newblock On a class of generating vector fields for the extremum seeking problem: Lie bracket approximation and stability properties.
\newblock {\em Automatica}, 94:151--160, 2018.

\bibitem{guay2021finite}
M.~Guay and M.~Benosman.
\newblock Finite-time extremum seeking control for a class of unknown static maps.
\newblock {\em International Journal of Adaptive Control and Signal Processing}, 35(7):1188--1201, 2021.

\bibitem{khalil2002nonlinear}
H.~K. Khalil.
\newblock {\em Nonlinear Systems}.
\newblock Prentice Hall, 2002.

\bibitem{krstic2000stability}
M.~Krstic and H.-H. Wang.
\newblock Stability of extremum seeking feedback for general nonlinear dynamic systems.
\newblock {\em Automatica}, 36(4):595--602, 2000.

\bibitem{lauand2022extremely}
C.~K. Lauand and S.~Meyn.
\newblock Extremely fast convergence rates for extremum seeking control with polyak-ruppert averaging.
\newblock {\em arXiv preprint arXiv:2206.00814}, 2022.

\bibitem{lauand2022markovian}
C.~K. Lauand and S.~Meyn.
\newblock Markovian foundations for quasi-stochastic approximation with applications to extremum seeking control.
\newblock {\em arXiv preprint arXiv:2207.06371}, 2022.

\bibitem{leblanc1922electrification}
M.~Leblanc.
\newblock Sur l'electrification des chemins de fer au moyen de courants alternatifs de frequence elevee.
\newblock {\em Revue g{\'e}n{\'e}rale de l'{\'e}lectricit{\'e}}, 12(8):275--277, 1922.

\bibitem{morosanov1957method}
I.~S. Morosanov.
\newblock Method of extremum control.
\newblock {\em Automatic and Remote Control}, 18:1077--1092, 1957.

\bibitem{oliveira2022extremum}
T.~R. Oliveira and M.~Krstic.
\newblock {\em Extremum Seeking Through Delays and PDEs}.
\newblock SIAM, 2022.

\bibitem{oliveira2016extremum}
T.~R. Oliveira, M.~Krsti{\'c}, and D.~Tsubakino.
\newblock Extremum seeking for static maps with delays.
\newblock {\em IEEE Transactions on Automatic Control}, 62(4):1911--1926, 2016.

\bibitem{pervozvanskii1960continuous}
A.~A. Pervozvanskii.
\newblock Continuous extremum control system in the presence of random noise.
\newblock {\em Automatic and Remote Control}, 21:673--677, 1960.

\bibitem{pokhrel2021control}
S.~Pokhrel and S.~A. Eisa.
\newblock Control-affine extremum seeking control with attenuating oscillations.
\newblock {\em arXiv preprint arXiv:2105.03985}, 2021.

\bibitem{poveda2021fixed}
J.~I. Poveda and M.~Krsti{\'c}.
\newblock Fixed-time seeking and tracking of time-varying extrema.
\newblock In {\em 2021 60th IEEE Conference on Decision and Control (CDC)}, pages 108--113. IEEE, 2021.

\bibitem{poveda2021non}
J.~I. Poveda and M.~Krstic.
\newblock Non-smooth extremum seeking control with user-prescribed fixed-time convergence.
\newblock {\em IEEE Transactions on Automatic Control}, 2021.

\bibitem{scheinker2021extremum}
A.~Scheinker, E.-C. Huang, and C.~Taylor.
\newblock Extremum seeking-based control system for particle accelerator beam loss minimization.
\newblock {\em IEEE Transactions on Control Systems Technology}, 30(5):2261--2268, 2021.

\bibitem{scheinker2014non}
A.~Scheinker and M.~Krsti{\'c}.
\newblock Non-c2 lie bracket averaging for nonsmooth extremum seekers.
\newblock {\em Journal of Dynamic Systems, Measurement, and Control}, 136(1), 2014.

\bibitem{scheinker2017model}
A.~Scheinker and M.~Krsti{\'c}.
\newblock {\em Model-free stabilization by extremum seeking}.
\newblock Springer, 2017.

\bibitem{song2017time}
Y.~Song, Y.~Wang, J.~Holloway, and M.~Krstic.
\newblock Time-varying feedback for regulation of normal-form nonlinear systems in prescribed finite time.
\newblock {\em Automatica}, 83:243--251, 2017.

\bibitem{song2023prescribed}
Y.~Song, H.~Ye, and F.~L. Lewis.
\newblock Prescribed-time control and its latest developments.
\newblock {\em IEEE Transactions on Systems, Man, and Cybernetics: Systems}, 2023.

\bibitem{sorensen2007controller}
K.~L. Sorensen, W.~Singhose, and S.~Dickerson.
\newblock A controller enabling precise positioning and sway reduction in bridge and gantry cranes.
\newblock {\em Control Engineering Practice}, 15(7):825--837, 2007.

\bibitem{suttner2019extremum}
R.~Suttner.
\newblock Extremum seeking control with an adaptive dither signal.
\newblock {\em Automatica}, 101:214--222, 2019.

\bibitem{suttner2017exponential}
R.~Suttner and S.~Dashkovskiy.
\newblock Exponential stability for extremum seeking control systems.
\newblock {\em IFAC-PapersOnLine}, 50(1):15464--15470, 2017.

\bibitem{tan2010extremum}
Y.~Tan, W.~H. Moase, C.~Manzie, D.~Ne{\v{s}}i{\'c}, and I.~M.~Y. Mareels.
\newblock Extremum seeking from 1922 to 2010.
\newblock In {\em Proceedings of the 29th Chinese control conference}, pages 14--26. IEEE, 2010.

\bibitem{tan2009global}
Y.~Tan, D.~Ne{\v{s}}i{\'c}, I.~M.~Y. Mareels, and A.~Astolfi.
\newblock On global extremum seeking in the presence of local extrema.
\newblock {\em Automatica}, 45(1):245--251, 2009.

\bibitem{todorovski2023practical}
V.~Todorovski and M.~Krstic.
\newblock Practical prescribed-time seeking of a repulsive source by unicycle angular velocity tuning.
\newblock {\em Automatica}, 154:111069, 2023.

\bibitem{wang2023delay}
J.~Wang and M.~Diagne.
\newblock Delay-adaptive boundary control of coupled hyperbolic pde-ode cascade systems.
\newblock {\em arXiv preprint arXiv:2301.12372}, 2023.

\bibitem{wang2022delay}
J.~Wang and M.~Krstic.
\newblock Delay-compensated event-triggered boundary control of hyperbolic pdes for deep-sea construction.
\newblock {\em Automatica}, 138:110137, 2022.

\bibitem{wang2016stability}
L.~Wang, S.~Chen, and K.~Ma.
\newblock On stability and application of extremum seeking control without steady-state oscillation.
\newblock {\em Automatica}, 68:18--26, 2016.

\bibitem{ctycdcpaper}
C.~T. Yilmaz, M.~Diagne, and M.~Krstic.
\newblock Exponential extremum seeking with unbiased convergence.
\newblock In {\em 62nd IEEE Conference on Decision and Control}. IEEE, 2023.

\bibitem{ctydelayheatPDETAC}
C.~T. Yilmaz and M.~Krstic.
\newblock Prescribed-time extremum seeking for delays and pdes using chirpy probing.
\newblock In {\em IEEE Transactions on Automatic Control, under review}. IEEE, 2022.

\bibitem{zhang2007extremum}
C.~Zhang, A.~Siranosian, and M.~Krsti{\'c}.
\newblock Extremum seeking for moderately unstable systems and for autonomous vehicle target tracking without position measurements.
\newblock {\em Automatica}, 43(10):1832--1839, 2007.

\end{thebibliography}

\appendix
\section{Proof of Theorem \ref{theoremrobust}}    \label{theoremappendix}

Let us proceed through the proof step by step.

\textbf{Step 1: State transformation.}
Consider the following transformations
\begin{align}
     \tilde{\theta}_f={}& \frac{1}{\alpha} \tilde{\theta}, \qquad
     \hat{G}_f={} \frac{1}{\alpha} \hat{G},  \qquad 
     \tilde{\eta}_f={} \frac{1}{\alpha^2} \tilde{\eta}, \label{transforrobust}
\end{align}
which transform \eqref{closedrobust} to the following system
\begin{align}
    &\frac{d}{d t}\begin{bmatrix}
    \tilde{\theta}_f & \hat{G}_f & \tilde{\eta}_f & \alpha
    \end{bmatrix}^T \nonumber \\
    &=\begin{bmatrix}
        \lambda \tilde{\theta}_f-\lambda \frac{\beta }{\alpha}\tilde{\theta}_f  +K \hat{G}_f \\ (\lambda-{\omega}_l-\lambda  \frac{\beta}{\alpha})\hat{G}_f +{\omega}_l\left[\nu(\tilde{\theta}_f \alpha+{S}(t)\alpha)-\tilde{\eta}_f \alpha^2 \right] \frac{{M}(t)}{\alpha^2}  \\
        (2\lambda-{\omega}_h-2\lambda  \frac{\beta}{\alpha}) \tilde{\eta}_f+{\omega}_h \frac{1}{\alpha^2} \nu(\tilde{\theta}_f \alpha+{S}(t)\alpha) \\
        -{\lambda} \alpha+\lambda \beta
    \end{bmatrix}, \label{robusttrans}
\end{align}
where $\nu$ is as defined in \eqref{nudefined} and satisfies \eqref{nupartials}.

\textbf{Step 2: Averaging operation.} 
We rewrite the system \eqref{robusttrans} in the time scale $\tau=\omega t$ as follows
\begin{align}
    &\frac{d}{d \tau}\begin{bmatrix}
    \tilde{\theta}_f & \hat{G}_f & \tilde{\eta}_f & \alpha
    \end{bmatrix}^T \nonumber \\
    &=\frac{1}{\omega} \begin{bmatrix}
        {\lambda} \tilde{\theta}_f-{\lambda} \frac{\beta }{\alpha}\tilde{\theta}_f+{K} \hat{G}_f \\ ({\lambda}-{\omega}_l-{\lambda}  \frac{\beta}{\alpha}) \hat{G}_f +{\omega}_l \left[\nu(\tilde{\theta}_f \alpha+\bar{S}(\tau)\alpha)-\tilde{\eta}_f \alpha^2 \right] \frac{\bar{M}(\tau)}{\alpha^2}  \\
        (2{\lambda}-{\omega}_h-2{\lambda}  \frac{\beta}{\alpha}) \tilde{\eta}_f +{\omega}_h \frac{1}{\alpha^2} \nu(\tilde{\theta}_f \alpha+\bar{S}(\tau)\alpha) \\
        -{\lambda} \alpha+{\lambda} \beta
    \end{bmatrix}. \label{robustdeltatrans}
\end{align}
where $\bar{S}(\tau)=S(\tau/\omega), \bar{M}(\tau)=M(\tau/\omega)$. The average of the system \eqref{robustdeltatrans} over the period $\Pi$ defined in \eqref{Pidef} is given by
\begin{align}
    \frac{d}{d \tau}\begin{bmatrix}
    \tilde{\theta}_f^a \\ \hat{G}_f^a  \\ \tilde{\eta}_f^a 
    \\ \alpha^a
    \end{bmatrix}={}&\frac{1}{\omega} \begin{bmatrix}
        ({\lambda}-{\lambda} \frac{\beta }{\alpha^a})\tilde{\theta}_f^a+{K} \hat{G}_f^a \\ ({\lambda}-{\omega}_l-{\lambda}  \frac{\beta}{\alpha^a}) \hat{G}_f^a \\
        (2{\lambda}-{\omega}_h-2{\lambda}  \frac{\beta}{\alpha^a}) \tilde{\eta}_f^a  \\
        -{\lambda} \alpha^a+\lambda \beta
    \end{bmatrix}  \nonumber \\
    +& \frac{1}{\omega}  \begin{bmatrix} 0 \\ {\omega}_l \frac{1}{\Pi} \int_0^{\Pi} \nu(\tilde{\theta}_f^a \alpha^a+\bar{S}(\sigma)\alpha^a)  \frac{\bar{M}(\sigma)}{(\alpha^a)^2} d\sigma \\ {\omega}_h \frac{1}{\Pi} \int_0^{\Pi}  \nu(\tilde{\theta}_f^a \alpha^a+\bar{S}(\sigma)\alpha^a) \frac{1}{(\alpha^a)^2} d\sigma \\ 0 \end{bmatrix}. \label{averagerobust}
\end{align}
It follows from \eqref{averagerobust} that the average equilibrium denoted as $\begin{bmatrix}
\tilde{\theta}_f^{a,e} & \hat{G}_f^{a,e} & \tilde{\eta}_f^{a,e} & \alpha^{a,e}
\end{bmatrix}^T$ satisfies
\begin{align}
    &\alpha^{a,e}={}\beta, \label{mueqrob} \\
    &\hat{G}_f^{a,e}={}0, \label{Gfeqrob} \\
    &\frac{{\omega}_l}{\Pi} \int_0^{\Pi} \nu(\tilde{\theta}_f^{a,e} \beta+\bar{S}(\sigma)\beta)  \frac{\bar{M}(\sigma)}{\beta^2} d\sigma={}0,  \label{nuMfeqrob} \\
   &\frac{1}{\Pi} \int_0^{\Pi} \nu(\tilde{\theta}_f^{a,e} \beta+\bar{S}(\sigma)\beta)  \frac{1}{\beta^2} d\sigma={} \tilde{\eta}_f^{a,e}. \label{aveqsatisfy}
\end{align}
Let us postulate the $i$th element $\tilde{\theta}_{f_i}^{a,e}$ of $\tilde{\theta}_{f}^{a,e}$ in the following form
\begin{align}
    \tilde{\theta}_{f_i}^{a,e}={}&\sum_{j=1}^n b_j^i a_j+\beta \sum_{j=1}^n \sum_{k \geq j }^n c_{j,k}^i a_j a_k+\mathcal{O}(\beta^2 |a^3|), \label{thetapost}
\end{align}
where $b_j^i$ and $c^i_{j,k}$ are real numbers, substitute \eqref{thetapost} into \eqref{nuMfeqrob}, perform a Taylor series expansion of $\nu$ as in \eqref{nutaylor} and equate the like powers of $a_j$. Then, we obtain $b_j^i=0$ $\forall i,j \in \{1,2\dots,n\}$, $c_{j,k}^i=0$ $\forall i,j, k \in \{1,2,\dots,n\}$ such that $j \neq k$ as well as
\begin{align}
    \begin{bmatrix}
    c_{j,j}^1 \\ \vdots \\ c_{j,j}^{i-1} \\ c_{j,j}^i  \\ c_{j,j}^{i+1} \\ \vdots \\ c_{j,j}^n
    \end{bmatrix}={}-H^{-1}  \begin{bmatrix}
    \frac{1}{4} \frac{\partial^3 \nu}{\partial z_1 \partial z_j^2}(0) \\ 
    \vdots \\
    \frac{1}{4} \frac{\partial^3 \nu}{\partial z_{j-1} \partial z_{j}^2}(0) \\ \frac{1}{8} \frac{\partial^3 \nu}{\partial z_j^3}(0) \\ \frac{1}{4}\frac{\partial^3 \nu}{\partial z_j^2 \partial z_{j+1}}(0) \\ \vdots \\ \frac{1}{4} \frac{\partial^3 \nu}{\partial z_j^2 \partial z_{n}}(0)
    \end{bmatrix}.
\end{align}
Following the same methodology, we compute $\tilde{\eta}_f^{a,e}$ in view \eqref{nutaylor}, \eqref{aveqsatisfy}, \eqref{thetapost}. Then, we find the equilibrium of the average system \eqref{averagerobust} as follows
\begin{align}
    \tilde{\theta}_{f_i}^{a,e}={}&\beta \sum_{j=1}^n c_{j,j}^i a_j^2+\mathcal{O}(\beta^2 |a|^3), \label{theequrob} \\
    \tilde{\eta}_f^{a,e}={}&\frac{1}{4} \sum_{i=1}^n H_{i,i} a_i^2+\mathcal{O}(\beta^2 |a|^4) \label{etaqurob}
\end{align}
together with \eqref{mueqrob} and \eqref{Gfeqrob}. 

\textbf{Step 3: Stability analysis.} The Jacobian of the average system \eqref{averagerobust} at  equilibrium is given by
\begin{align}
    &J_f^a \nonumber \\
    &=\frac{1}{\omega} \begin{bmatrix} 0_{n \times n} & {K} & 0_{n \times 1} & 0_{n \times 1}\\  \frac{{\omega}_l}{\Pi}  \int_0^{\Pi} \frac{\partial \left(   \frac{\nu\bar{M}}{(\alpha^{a})^2}  \right)}{\partial \tilde{\theta}_f}  d\sigma
    & -{\omega}_l I_{n \times n} & 0_{n \times 1} & \frac{{\omega}_l}{\Pi}  \int_0^{\Pi} \frac{\partial \left( \frac{\nu\bar{M}}{(\alpha^{a})^2}\right)}{\partial \alpha^a}  d\sigma   \\ \frac{{\omega}_h}{\Pi}  \int_0^{\Pi} \frac{\partial \left(   \frac{\nu}{(\alpha^{a})^2}  \right)}{\partial \tilde{\theta}_f}  d\sigma & 0_{1 \times n} & -{\omega}_h & \frac{{\omega}_h}{\Pi}  \int_0^{\Pi} \frac{\partial \left( \frac{\nu}{(\alpha^{a})^2}\right)}{\partial \alpha^a}  d\sigma \\ 0_{1 \times n} & 0_{1 \times n} & 0 & -{\lambda} \end{bmatrix}. \label{jacobrobust}
\end{align}
Considering the structure of \eqref{jacobrobust}, we get that it is Hurwitz if and only if the following inequality is satisfied 
\begin{align}
    \frac{{\omega}_l}{\Pi} \left( \int_0^{\Pi} \frac{\bar{M}(\sigma)}{(\alpha^{a,e})^2}  \frac{\partial }{\partial \tilde{\theta}_{f}}  \nu(\tilde{\theta}_f^{a,e} \alpha^{a,e}+\bar{S}(\sigma)\alpha^{a,e})  d\sigma \right)<0. \label{ineqrobust}
\end{align}
By performing a Taylor expansion of $\nu$ as in \eqref{nutaylor}, we get that \eqref{ineqrobust} is equal to ${\omega}_l H+[\mathcal{O}(\beta|a|)]_{n \times n}$. Let us define 
\begin{align}
    \mathcal{A}={}\begin{bmatrix} 0_{n \times n} & {K} \\  {\omega}_l H+[\mathcal{O}(\beta|a|)]_{n \times n}
    & -{\omega}_l I_{n \times n}\end{bmatrix},
\end{align}
which corresponds to $2n \times 2n$ submatrix in the upper left corner of \eqref{jacobrobust}. Then, we compute 
\begin{align}
    \text{det}(\lambda_{\mathcal{A}} I_{2n \times 2n}-(1/\omega) \mathcal{A} )={}&\text{det}\big((\lambda_{\mathcal{A}}^2+{\omega}_l (1/\omega) \lambda_{\mathcal{A}})I_{n \times n} \nonumber \\
    &-{\omega}_l (1/\omega^2) {K} H \nonumber \\
    &+[\mathcal{O}(\beta  (1/\omega^2) | a|)]_{n \times n}\big),
\end{align}
which, in view of $H < 0$, proves that $J_f^a$ is Hurwitz for 
sufficiently small $\beta |a|$. Then, based on the averaging theorem \cite{khalil2002nonlinear}, we show that there exist $\bar{\omega}, \bar{a}>0$ such that for all $\omega > \bar{\omega}$ and $\beta  |a| \in (0, \bar{a})$, the system has a unique exponentially stable periodic solution $(\tilde{\theta}_f^{\Pi}(\tau), \mu_f^{\Pi}(\tau), \tilde{\eta}_f^{\Pi}(\tau), \alpha^{\Pi}(\tau))$ of period $\Pi$ and this solution satisfies
\begin{align}
    \left|\tilde{\theta}_{f_i}^{\Pi}(\tau)-\beta \sum_{j=1}^n c_{j,j}^i a_j^2 \right| \leq {}&\mathcal{O}(1/\omega+\beta^2 |a|^3 ), \\
    \left| \hat{G}_f^{\Pi} (\tau) \right|  \leq {}&\mathcal{O}(1/\omega), \\
    \left| \tilde{\eta}_f^{\Pi}(\tau)-\frac{1}{4} \sum_{i=1}^n H_{i,i} a_i^2 \right| \leq {}& \mathcal{O}(1/\omega+\beta^2 |a|^4 ), \\
    \left| \alpha^{\Pi}(\tau)-\beta \right| \leq {}&\mathcal{O}(1/\omega).
\end{align}
In other words, all solutions $(\tilde{\theta}_f(\tau), \hat{G}_f(\tau),  \tilde{\eta}_f(\tau), \alpha(\tau))$ exponentially converge to a small neighborhood of the origin. The signal $\tilde{\theta}_f(\tau)$, in particular, converges to an $\mathcal{O}(1/\omega+\beta |a|^2)$-neighborhood of the origin.

\textbf{Step 4: Convergence to extremum.} Considering the results in Step 3 and recalling from \eqref{transforrobust} and Fig. \ref{ESBlock} with the modified $\alpha$-dynamics \eqref{mumodif} that 
\begin{align}
    \theta(t)={}&\alpha(t) \tilde{\theta}_f(t)+\theta^*+\alpha(t) S(t), \label{thetaconvrob}
\end{align}
we conclude the exponential convergence of $\theta(t)$ to an $\mathcal{O}(\beta/\omega+\beta |a| )$-neighborhood of $\theta^*$. Hence, performing a Taylor series expansion of $\nu$ as in \eqref{nutaylor} and noting \eqref{thetaconvrob}, we conclude the convergence of the output $y(t)$ to an $\mathcal{O}((\beta/\omega)^2+\beta^2 |a|^2)$-neighborhood of $h(\theta^*)$ and complete the proof of Theorem \ref{theoremrobust}.  $\hfill \blacksquare$

\end{document}